\nonstopmode
\documentclass[usenames,dvipsnames]{amsart}

\usepackage{amsmath}
\usepackage{amscd}
\usepackage{amssymb}
\usepackage[all]{xypic}
\usepackage{xcolor}
\usepackage{rotating}

\numberwithin{equation}{section}

\theoremstyle{plain} 
\newtheorem{theorem}[equation]{Theorem}
\newtheorem{lemma}[equation]{Lemma}

\newtheorem{proposition}[equation]{Proposition}
\newtheorem{corollary}[equation]{Corollary}

\newtheorem*{theoremapplication}{Theorem~\ref{theorem: application}}

\newcommand{\defining}[1]{{\emph{#1}}}  
\newcommand{\definedas}{:=}  

\theoremstyle{definition}
\newtheorem{definition}[equation]{Definition}

\newtheorem{remark}[equation]{Remark}
\newtheorem*{remark*}{Remark}

\newcommand{\integers}{{\mathbb{Z}}}

\newcommand{\complexes}{{\mathbb{C}}}
\newcommand{\field}{{\mathbb{F}}}

\newcommand{\pdash}{$p$\kern1.3pt-}

\newcommand{\op}{\rm{\,op}}
\newcommand{\cw}{\rm{cw}}

\newcommand{\LARGEstrut}{\mbox{{\LARGE\strut}}}
\newcommand{\Largestrut}{\mbox{{\Large\strut}}}
\newcommand{\largestrut}{\mbox{{\large\strut}}}

\newcommand{\tinystrut}{\mbox{{\tiny\strut}}}

\newcommand{\footnotestrut}{\mbox{{\footnotesize\strut}}}

\newcommand{\LnDiamond}{\Lcal_{n}^{\kern.2em\diamond}}
\newcommand{\Upk}{U\kern-2.5pt\left(p^{k}\right)}
\newcommand{\Upkup}{U(p^{k})}
\newcommand{\Uof}[1]{U\kern-2pt\left(#1\right)}
\newcommand{\Un}{\Uof{n}}
\newcommand{\Unup}{U(n)}

\newcommand{\theoremapplicationtext}{
Suppose that $M$ is a (co)Mackey functor for $\Un$ that takes values in \pdash local abelian groups and satisfies the \upanddown/ condition for the prime~$p$.
\begin{enumerate}
\item If $n$ is not a power of~$p$, then the map $\Lcal_n\to *$ induces an isomorphism on Bredon (co)homology with coefficients in $M$.
    \label{item: null homology}
\item If $n=p^k$, there is a map inducing an isomorphism on Bredon (co)homology with coefficients in~$M$:
    \label{item: homology approx by Tits}
\[
\Upk_{+}\wedge_{N\left(\Gamma_{k}\right)}\TitsSymp{k}^{\diamond}
\longrightarrow
\Lcal_{p^{k}}^{\diamond}.
\]
\end{enumerate}
}

\newtheorem*{propositionLn}{Proposition~\ref{proposition: L_n}}
\newcommand{\propositionLntext}{
    Let $p$ be a prime.  Let $E$ be a \pdash local spectrum with an action of~$\Un$.
     If $n$ is not a power of $p$, then $\HHwiggle_i^{U(n)}\left(\LnDiamond;
          \pibar^b_* E\right)=0$
     for all $i\ge 0$. If $n=p^k$, this holds for $i\ne k$. A similar statement holds for cohomology with coefficients in~$\piund^b_* E$.
    }

\newtheorem*{ApproximationTheorem}{Theorem~\ref{thm: approximation theorem}}
\newcommand{\ApproximationTheoremText}{
Let $G$ be a compact Lie group, and let $M$ be a (co)Mackey functor for $G$ that
satisfies the \upanddown\ condition for the prime~$p$ (Definition~\ref{definition: upanddown}).
Then for all $G$-CW-complexes $X$, the map $\alpha_{\footnotestrut X}: X_{\all{G}}\longrightarrow X$
is an $M$-(co)homology isomorphism.}

\newtheorem*{IsomorphismProposition}{Proposition~\ref{proposition: isomorphism}}

\newcommand{\IsomorphismPropositionText}{
Suppose that $f:X\rightarrow Y$ is a map of $G$-CW-complexes.
Assume that for all \pdash toral subgroups $H\subset G$, the spaces $X^H$ and $Y^H$ are homologically of finite type, and the map $X^{H}\longrightarrow Y^{H}$ is an isomorphism in mod~$p$ homology.
Let $M$ be a (co)Mackey functor satisfying the \upanddown/ condition for the prime~$p$
and taking values in \pdash local abelian groups.
Then $f$ induces an isomorphism in Bredon (co)homology with coefficients in~$M$.
}

\newtheorem*{PropositionCElmendorf}{Proposition~\ref{proposition: C-Elmendorf}}
\newcommand{\PropositionCElmendorfText}{
Let $X$ be a $G$-CW-complex and let $\Ccal$ be a collection of subgroups of~$G$. The natural transformation
$\alpha_{\footnotestrut X}: X_{\Ccal}\rightarrow X$
has the following properties.
\begin{enumerate}
\item If $\Iso(X)\subseteq\Ccal$, then $\alpha_{\footnotestrut X}$ is a weak $G$-equivalence.
     \label{item: equivalence}
\item $\Iso(X_\Ccal)\subseteq \Ccal$.
     \label{item: isotropy groups}
\item If $H\in \Ccal$, then $\alpha_{\footnotestrut X}$ induces a weak equivalence
     $\left(X_\Ccal\right)^H \to X^H$.
     \label{item: fixed point equivalence}
\end{enumerate}
The map $\alpha_{\footnotestrut X}: X_\Ccal\rightarrow X$  is characterized up to a weak $G$-equivalence by properties \eqref{item: isotropy groups} and~\eqref{item: fixed point equivalence}.
}

\def\doCal#1{%
\ifx#1\doAllCalEnd\def\doAllCal{\relax}\else%
 \expandafter\edef\csname#1cal\endcsname{{\noexpand\mathcal #1}}\fi}
\def\doAllCal#1{\doCal#1\doAllCal}
\doAllCal ABCDEFGHIJHLMNOPQRSTUVWXYZ\doAllCalEnd

\def\doBar#1{%
\ifx#1\doAllBarEnd\def\doAllBar{\relax}\else%
 \expandafter\edef\csname#1bar\endcsname{{\noexpand\overline{#1}}}\fi}
\def\doAllBar#1{\doBar#1\doAllBar}
\doAllBar ABCDEFGHIJHLMNOPQRSTUVWXYZabcdefghijhlmnopqrstuvwxyz\doAllBarEnd
\newcommand{\pibar}{{\overline{\pi}}}
\newcommand{\piund}{{\underline{\pi}}}

\def\doWiggle#1{%
\ifx#1\doAllWiggleEnd\def\doAllWiggle{\relax}\else%
 \expandafter\edef\csname#1wiggle\endcsname{{\noexpand\tilde{#1}}}\fi}
\def\doAllWiggle#1{\doWiggle#1\doAllWiggle}
\doAllWiggle ABCDEFGHIJHLMNOPQRSTUVWXYZabcdefghijklmnopqrstuvwxyz\doAllWiggleEnd

\newcommand{\whatever}{\text{---\,}}

\DeclareMathOperator{\ad}{ad}

\DeclareMathOperator{\Aut}{Aut}

\DeclareMathOperator{\cokernel}{cok}

\newcommand{\CoEnd}{{\mbox{coend}}}
\newcommand{\End}{{\mbox{end}}}

\newcommand{\hocolim}{\operatorname{hocolim}\,}

\DeclareMathOperator{\Iso}{Iso}
\DeclareMathOperator{\map}{map}

\DeclareMathOperator{\Sp}{Sp}

\DeclareMathOperator{\Syl}{Syl}

\DeclareMathOperator{\tr}{tr}

\newcommand{\Symp}[1]{\Sp_{#1}\left(\field_{p}\right)}
\newcommand{\TitsSymp}[1]{T\Sp_{#1}\left(\field_{p}\right)}
\newcommand{\SO}[1]{\Scal\!\left(\Ocal_{#1}\right)}
\newcommand{\nat}{\operatorname{nat}}
\newcommand{\hnat}{\operatorname{hnat}}

\renewcommand{\top}{\mbox{Top}}
\newcommand{\ch}{\mbox{Ch}}

\newcommand{\nc}{C_*}

\newcommand{\htensor}{\otimes^h}
\newcommand{\HH}{\operatorname{H}}
\newcommand{\HHwiggle}{\widetilde{\HH}}
\newcommand{\category}{{\bf{C}}}
\newcommand{\orbit}{O}
\newcommand{\FixedPtFunctor}[1]{\Phi_{#1}}
\newcommand{\unit}{{\mathtt{1}}}
\newcommand{\derCoEnd}[1]{\boxtimes^h_{#1}}

\newcommand{\specialorbit}{Q}

\newcommand{\DoubleLeftArrow}{\ \substack{
        \longleftarrow \\[-1em]
        \longleftarrow
        }
}

\newcommand{\DoubleRightArrow}{\ \substack{
        \longrightarrow \\[-1em]
        \longrightarrow
        }
}
\newcommand{\SimpOneToZero}{\substack{
        \longleftarrow \\[-0.8em]
        \longrightarrow \\[-0.8em]
        \longleftarrow
        }
}
\newcommand{\CosimpZeroToOne}{\substack{
        \longrightarrow \\[-0.8em]
        \longleftarrow \\[-0.8em]
        \longrightarrow
        }
}

\newcommand{\SimpTwoToOne}{\substack{
       \longleftarrow \\[-0.8em]
       \longrightarrow \\[-0.8em]
       \longleftarrow \\[-0.8em]
       \longrightarrow \\[-0.8em]
       \longleftarrow
       }
}

\newcommand{\CosimpOneToTwo}{\substack{
       \longrightarrow \\[-0.8em]
       \longleftarrow \\[-0.8em]
       \longrightarrow \\[-0.8em]
       \longleftarrow \\[-0.8em]
       \longrightarrow
       }
}

\newcommand{\all}[1]{\Acal_{#1}}

\newcommand{\restrictedto}[1]{\!\!\left.\right|_{#1}}
\newcommand{\hobased}{{\tilde{h}}}
\newcommand{\MacLane}{Mac\,Lane }

\newcommand{\sz}[1]{\operatorname{size}\left({#1}\right)}

\def\Text#1{\def\TextString{#1}\futurelet\TextDelim\TextSkip}
\def\TextSkip{\ifx\TextDelim/\def\TextDo{\TextString\EatOne}%
              \else\let\TextDo\TextString\fi%
              \TextDo}
\def\EatOne#1{}

\def\SkipToEndScan#1\EndScan{}
\def\Scan#1#2#3{\ifx#1#2#3\expandafter\SkipToEndScan\fi\Scan#1}
\def\Upper#1{%
\Scan#1aAbBcCdDeEfFgGhHiIjJkKlLmMnNoOpPqQrRsStTuUvVwWxXyYzZ#1#1\EndScan}
\def\Phrase#1 #2/#3/#4=#5 #6/#7/#8.{%
\expandafter\edef\csname#2#3\endcsname{\noexpand\Text{#6#7}}
\expandafter\edef\csname\Upper#2#3\endcsname{\noexpand\Text{\Upper#6#7}}
\expandafter\edef\csname#1#2#3\endcsname{\noexpand\Text{#5 #6#7}}
\expandafter\edef\csname\Upper#1#2#3\endcsname{\noexpand\Text{\Upper#5 #6#7}}
\expandafter\edef\csname#2#4\endcsname{\noexpand\Text{#6#8}}
\expandafter\edef\csname\Upper#2#4\endcsname{\noexpand\Text{\Upper#6#8}}
}

\Phrase an upanddown//s=an transfer//s.
\Phrase an size//s=an size//s.
\Phrase a Ktheory//s=a $K$-theory//s.
\Phrase a Ktheoretic//s=a $K$-theoretic//s.

\author{Gregory Arone}
\address{Stockholm University}
\email{gregory.arone@math.su.se}
\author{W. G. Dwyer}
\address{Department of Mathematics, University of Notre Dame,
             Notre Dame, IN 46556 USA}
\email{dwyer.1@nd.edu}
\author{Kathryn Lesh}
\address{Department of Mathematics, Union College, Schenectady NY}
\email{leshk@union.edu}

\thanks{The first author was partially supported by
Swedish Research Council grant 2016-05440}


\begin{document}

\title{\pdash toral approximations compute Bredon homology}

\begin{abstract}
We study Bredon homology approximations for spaces with an action of a compact Lie group~$G$. We show that if $M$ is a coMackey functor satisfying mild \pdash locality conditions, then Bredon homology of a $G$-space $X$ with coefficients in $M$ is determined by fixed points of \pdash toral subgroups of $G$ acting on~$X$. As an application we prove a vanishing result for the Bredon homology of the complex $\Lcal_{n}$ of direct-sum decompositions of~$\complexes^{n}$.
\end{abstract}

\maketitle

\section{Introduction}
Let $X$ be a space with an action of a compact Lie group~$G$, and let $p$ be a fixed prime. The
$G$-equivariant homotopy type of $X$, and hence its Bredon
homology\footnote{In general, we are interested in both homology and cohomology, so we write ``(co)homology" throughout the paper to indicate both together. To avoid cluttering the exposition, we will focus on homology in the introduction.},
is determined by the fixed point spaces $X^H$ of all closed subgroups $H\subseteq G$. In this paper, we establish that the Bredon homology of $X$ can sometimes be computed from knowledge of $X^H$ for only a subset of the subgroups of~$G$.
Our main results in this direction say that if a coefficient system $M$ comes from a coMackey functor and satisfies certain \pdash locality conditions,
then Bredon homology of~$X$ with coefficients in $M$ can be determined from the mod~$p$ homology of fixed point spaces of just the \pdash toral subgroups of $G$.

Our main application concerns the Bredon homology of $\Lcal_n$, the complex of proper direct-sum decompositions of $\complexes^n$, which is a finite complex with an action of~$\Un$. The space $\Lcal_{n}$ was introduced in~\cite{Arone-Topology} and studied in detail, largely from first principles, in \cite{Banff1} and~\cite{Banff2}.
Let $\TitsSymp{k}$ denote the symplectic Tits building (see Section~\ref{section: applications}), and let $X^{\diamond}$ denote the unreduced suspension of~$X$. Let $\Gamma_{k}\subset\Upk$ denote the unique subgroup (up to conjugacy) that acts irreducibly on~$\complexes^{p^k}$ and is an extension of the central $S^{1}\subset\Upk$ by an elementary abelian \pdash group.
We use Smith theory to leverage the results of~\cite{Banff2} and prove
our principal concrete result, the following theorem for~$\Lcal_{n}$,
in Section~\ref{section: applications}.
(See Definition~\ref{definition: upanddown} for the \upanddown/ condition.)

\begin{theoremapplication}
\theoremapplicationtext
\end{theoremapplication}

The equivariant homotopy type of $\LnDiamond$ came up in the first author's work on orthogonal calculus~\cite{Arone-Topology} and in the first and third author's joint work on the rank filtration~\cite{Arone-Lesh-Crelle, Arone-Lesh-Fundamenta}. It has also played a role in the study of the Balmer spectrum of the equivariant stable homotopy category~\cite{Balmer-spectrum}.
The results of our current work will be used in a forthcoming paper on the \Ktheoretic\ analogue of the Whitehead conjecture.

Our results and methods are analogous to those used in \cite{ADL2}, where we proved a result similar to Theorem~\ref{theorem: application} for the
$\Sigma_{n}$-space~$\Pcal_{n}$, the poset of nontrivial, proper partitions of
the set $\{1,...,n\}$. One expects this analogy because of the dictionary provided by
\cite{Arone-Topology}, which gives a correspondence between objects associated to the Goodwillie tower for the identity (the ``discrete case") and the Weiss tower for the
functor $V\mapsto B\Aut(V)$ (the ``orthogonal" or ``unitary" or ``compact Lie" case).
As is standard for translation between finite groups and compact Lie groups, the role of finite \pdash groups in the discrete case is taken by \pdash toral groups (extensions of a torus by a finite \pdash group) in the compact Lie case. In the remainder of the introduction, we outline our methods and other results in more detail. As is often the case, some things are harder in the compact Lie case than in the finite case, and others turn out to be easier when the compact Lie group is connected.

The first step of \cite{ADL2}, when $G$ was a finite group, relied on approximating a $G$-space~$X$ by a space that uses only \pdash subgroups of~$G$ as isotropy groups.
A~similar construction of an approximation in the compact Lie case is not difficult once we set up the appropriate language (Section~\ref{section: preliminaries}).  Let $\Ccal$ be a collection of (closed) subgroups of~$G$, that is, a set of subgroups of $G$ that is closed under conjugation. In Section~\ref{section: preliminaries}, we construct a functorial $G$-CW-complex~$X_{\Ccal}$, equipped with a natural map $\alpha_{X}: X_{\Ccal}\to X$, which is an approximation of $X$ by a space whose isotropy is contained in~$\Ccal$. Let $\Iso(X)$ denote the collection of isotropy subgroups of the action of $G$ on~$X$.

\begin{PropositionCElmendorf}
\PropositionCElmendorfText
\end{PropositionCElmendorf}

Turning our attention to the prime~$p$, let $\all{G}$ be the collection of all \pdash toral subgroups of~$G$ (including the trivial subgroup).
We would like to know when $X_{\all{G}}\longrightarrow X$ is an isomorphism
on Bredon homology, so that we can reduce computing Bredon homology of $X$ to computing Bredon homology of~$X_{\all{G}}$.
When $G$ is discrete, establishing such an isomorphism involves having a coMackey functor for the coefficient system, satisfying a suitable transfer condition, and then making a standard chain-level transfer argument.

In the compact Lie case, chain methods have to be replaced with stable homotopy-theoretic ones. Because of this,
we need to revisit the precise definitions of coMackey functors and Bredon homology. Moreover, it is another feature of the compact Lie case that we must consider separately coMackey functors (for taking homology) and Mackey functors (for taking cohomology). In the introduction we focus on the case of coMackey functors and homology, but each statement has a corresponding statement involving Mackey functors and cohomology.
This is discussed systematically in Sections~\ref{section: Bredon homology} and~\ref{section: reduction to p-Sylow}.

An appropriate \pdash locality condition is required on a
coefficient system in order for us even to hope that $X_{\all{G}}$ is sufficient for computational purposes.
Bredon homology with coefficients in a coMackey functor $M$ can be extended to an $RO(G)$-graded homology theory,
and this phenomenon provides the existence of transfer
homomorphisms.
Let $P$ be a maximal \pdash toral subgroup of $G$.
We say that $M$ satisfies the \defining{\upanddown/ condition for~$p$} if for every
subgroup $H\subseteq G$, the composition
\[
\HH^G_*(G/H; M)
       \xrightarrow{\ \tr\ } \HH^G_*(G/P\times G/H; M)
       \longrightarrow \HH^G_*(G/H; M)
\]
is an isomorphism (Definition~\ref{definition: upanddown}). Here $\HH_{*}^{G}(\whatever; M)$ denotes Bredon homology with coefficients in the coMackey functor~$M$. The first homomorphism is induced by the Becker-Gottlieb transfer associated with the projection $G/P\times G/H \to G/H$.

\begin{remark*}
Notice that $\HH^G_*(G/H; M)\cong M(G/H)$. When $G$ is a finite group, the product $G/P \times G/H$ is a disjoint union of $G$-orbits, so $M(G/P \times G/H)$ can be defined using additivity. In that case the \upanddown/ condition can be defined as a requirement that the composite homomorphism $M(G/H)\to M(G/P \times G/H) \to M(G/H)$ is an isomorphism for all subgroups~$H\subseteq G$. But when $G$ is a compact Lie group, $G/P \times G/H$ is usually not a disjoint union of $G$-orbits, and therefore $M(\whatever)$ needs to be replaced with $\HH^G_*(\whatever; M)$ in the formulation of the condition.
\end{remark*}

The following are two important general results of this %
paper.\footnote{We apologize to the reader that the parentheses in ``(co)homology" and ``(co)Mackey" do not correspond, but rather are reversed.
The reason is explained in a comment after Definition~\ref{definition: Mackey and coMackey}.
}

\begin{ApproximationTheorem}
\ApproximationTheoremText
\end{ApproximationTheorem}

\begin{IsomorphismProposition}
\IsomorphismPropositionText
\end{IsomorphismProposition}

Section~\ref{section: applications} introduces the first application to the $\Un$-space $\Lcal_n$. In this section we prove Theorem~\ref{theorem: application}, stated at the outset of the introduction, by first using
Theorem~\ref{thm: approximation theorem} to approximate $\Lcal_{n}$ using \pdash toral subgroups of~$\Un$. The approach is analogous to that of \cite{ADL2}, where we began by approximating the partition complex $\Pcal_{n}$ using \pdash subgroups of~$\Sigma_{n}$.
In both \cite{ADL2} and in the present work, it is necessary to find an argument to ``discard" a subgroup from the approximating collection if that subgroup's fixed point space is not contractible. In~\cite{ADL2}, the
technology used is a ``pruning" argument, based on homological properties of
the poset of nontrivial \pdash subgroups of the Weyl group of the problematic subgroup.
The argument had to be applied to a large number of \pdash subgroups of $\Sigma_{n}$ that might have
 non-contractible fixed point spaces on~$\Pcal_{n}$.

In the case of $\Un$ acting on~$\Lcal_{n}$, there can likewise be \pdash toral subgroups of~$\Un$ whose fixed point spaces are not contractible and need to be handled.
However, it is a pleasant feature of $\Lcal_{n}$ that it is acyclic (indeed, contractible) unless $n$ is a power of a prime. It follows that we can use Smith
theory to conclude that the fixed point space $\Lcal_n^P$ is mod~$p$ acyclic whenever $P$ is a \pdash toral subgroup of~$\Un$ and $n$ is a not a prime power. This allows us to disregard potentially problematic groups without needing to resort to an elaborate pruning argument. When $n=p^k$, the space $\Lcal_{p^k}$ is not acyclic, but it is \pdash locally equivalent to another space that is in some sense much simpler, and again Smith theory can be used.

It seems plausible that the pruning method of \cite{ADL2} could also be adapted to the present situation, but we decided not to go down that road.

Section~\ref{section: up and down} begins to deal with the question of which (co)Mackey functors satisfy the \upanddown/ condition for~$p$. Our main result along these lines is Proposition~\ref{proposition: up and down}, which says that if a generalized homology theory satisfies the \upanddown/ condition, then so does Bredon homology with coefficients in the associated coMackey functor.

For future applications, we are especially interested in coMackey functors that come from Borel homology, and we discuss this topic in Section~\ref{section: borel}.
More precisely,  let $E$ be a spectrum with an action of~$G$, and
let $\hobased$ denote based homotopy orbits.
Let $\pibar_{*}^{b}E$ be the coMackey functor defined by the formula
$(\pibar_*^b E)(\Sigma^\infty \orbit_+)
      =\pi_*(E\wedge O_+)_{\hobased G}$.
We show that if $E$ is (non-equivariantly) \pdash local,
then the coMackey functor $\pibar_{*}^{b}E$ satisfies the \upanddown/ condition for~$p$ (Proposition~\ref{proposition: p-local}).
This is a result we need for forthcoming applications.

Finally, in Section~\ref{section: LnBorel} we combine the results of Sections~\ref{section: up and down} and~\ref{section: borel} to obtain the following computation for~$\Lcal_{n}^{\diamond}$.
Let $\HHwiggle_*$ denote reduced Bredon homology.

\begin{propositionLn}
\propositionLntext
\end{propositionLn}

\medskip
\noindent{\bf{Organization and notation:}}\\
\indent
Section~\ref{section: preliminaries} reviews some basic facts regarding actions of compact Lie groups on spaces, such as Elmendorf's theorem, and approximation relative to a collection of subgroups.

Section~\ref{section: Bredon homology} reviews coefficient systems, and Bredon homology and cohomology of spaces with an action of a compact Lie group.
In Section~\ref{section: reduction to p-Sylow}, we discuss (co)Mackey functors for compact Lie groups, and equivariant homology and cohomology with coefficients in (co)Mackey functors.
We then show how to use the transfer to reduce from a compact Lie group to a maximal \pdash toral subgroup.

Section~\ref{section: toral} states and proves our main approximation results,
Theorem~\ref{thm: approximation theorem} and Proposition~\ref{proposition: isomorphism}. In Section~\ref{section: applications} we apply our results to prove Theorem~\ref{theorem: application}, a general result about the Bredon (co)homology of the complex $\Lcal_{n}$ of direct-sum decompositions, with coefficients in a rather general (co)Mackey functor.

Section~\ref{section: up and down} establishes that if a generalized equivariant homology theory $E_*^G$ satisfies the \upanddown/ condition, then so does homology with coefficients in the coMackey functor associated with~$E_*^G$. In Section~\ref{section: borel} we review the Borel homology theory associated with a $G$-spectrum, and verify that if a $G$-spectrum is \pdash local then the associated Borel homology
satisfies the \upanddown/ condition for~$p$.

Finally, in Section~\ref{section: LnBorel} we apply our results to compute
a vanishing result for the Bredon homology of the complex $\Lcal_{n}$ of direct-sum decompositions with coefficients in a coMackey functor associated with Borel homology. This will be used in a future work on the \Ktheoretic/ analogue of the Whitehead conjecture.

\smallskip
\noindent{\bf{Terminology:}}\\
\indent
Throughout the paper, $p$ is a fixed prime, ``group'' means compact Lie group, and ``subgroup'' means closed subgroup. A ``$G$-space'' $X$ means a $G$-CW complex, and we write $\Iso(X)$ to indicate the collection of isotropy subgroups of the action of $G$ on~$X$.
Unless otherwise indicated,
the word ``spectrum'' means a $G$-equivariant spectrum indexed on a
complete $G$-universe. In particular, phrases like ``suspension
spectrum,'' ``Eilenberg-\MacLane spectrum,'' etc., refer to the
$G$-equivariant versions of such concepts. A ``collection of subgroups" of a group~$G$ means a set of subgroups that is closed under conjugation.

\smallskip
\noindent{\bf{Acknowledgements:}}\\
\indent The third author thanks Vesna Stojanoska for preliminary conversations on this topic.
Our thanks to C.~Malkiewich for pointing out a helpful reference. Lastly, we are grateful
to the referee for a careful reading of the initial submission.

\section{$G$-spaces as diagram categories}\label{section: preliminaries}
Let $G$ be a compact Lie group. In this section we collect some background about the homotopy theory of $G$-spaces, including the construction of an equivariant approximation relative to a collection of subgroups. First, the weak equivalences.

\begin{definition}
Let $G$ be a compact Lie group, and let $X$ and $Y$ be $G$-spaces. A~$G$-equivariant map $X\rightarrow Y$ is a \defining{weak $G$-equivalence} if for every subgroup $H\subseteq G$, the map of fixed point spaces $X^H\rightarrow Y^H$ is a weak equivalence of spaces.
\end{definition}

Let $\Ocal_G$ denote the orbit category of~$G$, whose objects are $G$-orbits and
whose morphisms are $G$-equivariant maps. The category $\Ocal_{G}$  is a topologically enriched category: the morphism sets are naturally topologized as subspaces
of orbits of~$G$. Elmendorf's seminal work~\cite{Elmendorf} established that the homotopy category of $G$-spaces is equivalent to the homotopy category of continuous functors $\Ocal_G^{\op} \to \top$.
In this section we review Elmendorf's construction with appropriate language for our needs, and extend it to spaces whose isotropy groups are restricted to a collection of subgroups.
We use this generalization to establish the criterion, standard for finite groups, for a map to be a weak $G$-equivalence (Corollary~\ref{corollary: detect}).

Elmendorf's equivalence of categories is induced by
the following ``fixed-points" functor.

\begin{definition}
Given a $G$-space~$X$, let $\FixedPtFunctor{X}$ be the functor
$\Ocal_{G}^{\op}\to \top$ given by the formula
$\FixedPtFunctor{X}(O)=\map_G(O, X)$.
\end{definition}

We may view $\FixedPtFunctor{}$ as a functor from $G$-spaces (i.e. $X$) to the category of continuous functors $\Ocal_G^{\op} \to \top$ (i.e. $\FixedPtFunctor{X}$). By Elmendorf's theorem, $\FixedPtFunctor{}$ induces an equivalence of homotopy categories. Further,
$\FixedPtFunctor{}$ is actually the right adjoint of a Quillen equivalence~\cite{Piacenza}.
The (derived) left adjoint of $\Phi$ can be constructed as a (derived) enriched \CoEnd,
as we discuss below.

\subsubsection*{Enriched categories and functors}
Let $(\Vcal, \boxtimes, \unit)$ be a closed symmetric monoidal category with
small limits and colimits. Our prime examples of categories $\Vcal$ are the
category~$\top$ (topological spaces) and the category~$\ch_{\integers}$
(chain complexes of abelian groups).
Suppose $\category$ is a category enriched over~$\Vcal$. Thus $\category$ consists of a set of objects, and for every two objects $x,y$, the morphisms from $x$ to $y$ are given by an object of $\Vcal$, denoted $\category(x,y)$. There are unit morphisms $\unit\to \category(x, x)$ for every object~$x$, and composition morphisms $\category(x, y)\boxtimes\category(y, z)\to \category(x, z)$
for every triple of objects, and these structure maps are associative and unital.
An enriched functor
$F\colon \category \to \Vcal$ associates to every object $x$ of $\category$ an object $F(x)$ of $\Vcal$; to any two objects $x,y$ of~$\category$,
the functor $F$ associates a $\Vcal$-morphism
$F(x)\boxtimes \category(x,y) \to F(y)$. Once again, these structure maps
are required to be associative,
and unital.
There is a similar definition for an
enriched functor $G\colon \category^{\op} \to \Vcal$, with the structure morphisms having the form $\category(x,y) \boxtimes G(y) \to G(x)$.

\begin{remark*}
When $\Vcal$ is a nonspecific symmetric monoidal category, we will use the symbol $\boxtimes$ to denote the monoidal product. When $\Vcal$ is a specific category, however, we use the standard notation for the symmetric monoidal product in that category. Thus when $\Vcal=\top$
we use~$\times$, and when $\Vcal=\ch_{\integers}$ we
use~$\otimes$.
\end{remark*}

\subsubsection*{Enriched constructions}
Next we look at enriched $\CoEnd$ and $\End$ (natural transformations).
Functors $F\colon \category \to \Vcal$ and
$G\colon \category^{\op} \to \Vcal$
are analogous to a right and a left module,
respectively, over a ring.
The enriched $\CoEnd$ of $F$ and $G$ is analogous to the tensor product of modules. Dually, the enriched $\End$ is analogous to $\hom$ of modules.

The enriched $\CoEnd$ of $F$ and~$G$, denoted $F\boxtimes_{\category} G$ is defined by the usual coequalizer diagram in~$\Vcal$:
\[
F\boxtimes_{\category} G\longleftarrow \bigoplus_{x_0} F(x_0)\boxtimes G(x_0)
\DoubleLeftArrow \bigoplus_{x_0, x_1} F(x_0)\boxtimes \category(x_0, x_1)\boxtimes G(x_1).
\]

Dually, if $F, G\colon \category \to \Vcal$ are both enriched functors, then the enriched $\End$ of $F$ and~$G$ (sometimes called the object of enriched natural transformations from $F$ to~$G$) is defined by the following equalizer diagram, where $\hom_{\Vcal}$ denotes the internal hom object in~$\Vcal$:
\[
\nat_{\category}(F, G)
\longrightarrow \prod_{x_0} \hom_\Vcal\left(\strut F(x_0), G(x_0)\right)
\DoubleRightArrow \prod_{x_0, x_1}
   \hom_\Vcal\left(\strut F(x_0)\boxtimes \category(x_0, x_1), G(x_1)\right).
\]

\subsubsection*{Derived constructions}
Adding the next piece of structure, we suppose that $\Vcal$ is actually a closed symmetric monoidal
{\emph{Quillen model category}}. Limits and colimits do not always preserve weak equivalences. Therefore in the homotopy setting one often needs to work with {\emph{derived}} (``homotopy") limits and colimits. These constructions constitute universal approximations to the strict limits and colimits by homotopy invariant constructions. We will focus on derived $\CoEnd$ and $\End$ in particular.

Under mild
assumptions\footnote{
See~\cite[Theorem 5.4]{Piacenza} for the case of topological spaces and~\cite[Proposition 6.3]{Schwede-Shipley} for some general conditions.
}
on $\Vcal$, the category of $\Vcal$-enriched functors with a fixed set of objects is equipped with the projective Quillen model structure, i.e. the structure where fibrations and weak equivalences are determined level-wise.
One way to define the derived $\CoEnd$ of $F$ and $G$ is by taking the strict $\CoEnd$ of cofibrant replacements of $F$ and~$G$. Dually, the derived end of $F$ and $G$ is the strict end of a cofibrant replacement of $F$ and a fibrant replacement of~$G$. This is the viewpoint taken, for example, in~\cite{ADL1}. In the current work, however, we choose to define the derived $\CoEnd$/$\End$ via the bar/cobar constructions. This definition is sometimes convenient
for calculations, and is also well suited for proving specialized invariance
results such as Lemma~\ref{lemma: fixed points iso implies bredon}
later on. For detail beyond what is given in
Definition~\ref{defn: End and CoEnd} and Remark~\ref{remark: realization},
we refer the reader to~\cite{Shulman, Riehl-Categorical}.

\begin{definition}    \label{defn: End and CoEnd}
\hfill
\begin{enumerate}
\item  \label{item: der coend}
Suppose that
$F\colon \category \to \Vcal$ and
$G\colon \category^{\op} \to \Vcal$
are enriched functors.
We assume that all the objects $\category(x,y)$ are cofibrant in~$\Vcal$, and that
$F$ and $G$ are objectwise cofibrant. The derived $\CoEnd$
of $F$ and~$G$ will be denoted by $F\derCoEnd{\category}G$, or sometimes by $F(x)\derCoEnd{x}G(x)$, or by $F(x)\derCoEnd{x\in \category}G(x)$, depending on what needs to be emphasized. It is defined to be
the realization of the following simplicial object in $\Vcal$:
\begin{multline*}
\bigoplus_{x_0} F(x_0)\boxtimes G(x_0)
\,\SimpOneToZero\,
\bigoplus_{x_0,x_1} F(x_0)\boxtimes\category(x_0,x_1)\boxtimes G(x_1)
\\
\SimpTwoToOne
\bigoplus_{x_0,x_1,x_2}
    F(x_0)  \boxtimes\category(x_0,x_1)
            \boxtimes\category(x_1,x_2)
            \boxtimes G(x_2)
\,\cdots
\end{multline*}

\item   \label{item: der end}
Suppose $F$ and $G$ are both enriched functors from
$\category$ to~$\Vcal$.
Assume that $F$ is
objectwise cofibrant and $G$ is objectwise fibrant. The derived natural transformations
from $F$ to $G$, denoted $\hnat_{\Vcal}(F, G)$, is the totalization
of the following cosimplicial object in $\Vcal$:
\begin{multline*}
\prod_{x_0} \hom_{\Vcal}\left(\strut F(x_0), G(x_0)\right)
\,\CosimpZeroToOne\,
\prod_{x_0,x_1} \hom_{\Vcal}\left(\strut F(x_0)\boxtimes\category(x_0,x_1),G(x_1)\right)
\\
\CosimpOneToTwo\,
\prod_{x_0,x_1,x_2}
     \hom_\Vcal\left(\strut F(x_0)\boxtimes\category(x_0,x_1)
                \boxtimes\category(x_1,x_2), G(x_2)\right)
\, \cdots
\end{multline*}
\end{enumerate}
\end{definition}

\begin{remark}  \label{remark: realization}
By geometric realization of a simplicial object (resp. the totalization of a cosimplicial object) in~$\Vcal$, we mean the homotopy colimit (resp. homotopy limit) of the underlying diagram, taken in~$\Vcal$. Homotopy limits and colimits can be defined in any model category. They are only defined up to a natural weak equivalence, but it does not really matter which model we use.

In some cases there is a particularly nice model for the geometric realization. For example, suppose that $\Vcal=\top$ and that for every object $x$ of~$\category$, the mapping space $\category(x, x)$ is well pointed by the identity map; then the simplicial bar resolution of $F\times^h_\Ccal G$ is a Reedy cofibrant simplicial space, and one may use the classic construction of geometric realization. For another example, if $\Vcal=\ch_{\integers}$, then a good model for the geometric realization/totalization is obtained by first applying the
normalized chains functor to obtain a bicomplex, and then taking the total complex. See~\cite[Proposition 16.9]{dugger} for the simplicial case, and~\cite[Lemma 2.2]{bousfield} for the cosimplicial case.
\end{remark}

The derived constructions above come equipped with a derived (co)Yoneda lemma, as we describe now.
Suppose $z$ is a fixed object of~$\category$.
Let $\category(\whatever, z)$ be the contravariant functor represented by~$z$, i.e., the functor $x\mapsto\category(x,z)$. The following is a well-known
version of the derived enriched (co)Yoneda lemma. As we will see, it is a basic tool for proving facts about the derived $\CoEnd$ and $\End$
(for example, Theorem~\ref{theorem: elmendorf} and Proposition~\ref{proposition: C-Elmendorf}).

\begin{lemma}\cite[Example 4.5.7, Theorem 5.1.1]{Riehl-Categorical}\label{lemma: coyoneda}
Let $F$ be a covariant functor enriched over~$\category$.
The evaluation maps
$F(x) \boxtimes \category(x, z) \to F(z)$ induce a natural weak equivalence in $\Vcal$
\[
 F(x) \boxtimes^h_{x\in\category} \category(x, z)
    \xrightarrow{\ \simeq\ } F(z). \tag{coYoneda lemma}
\]
Dually, there is a natural weak equivalence
\[
F(z)\xrightarrow{\ \simeq\ }
     \hnat_{\category}\left(\strut\category(z, \whatever), F(\whatever)\right). \tag{Yoneda lemma}
\]
\end{lemma}

\bigskip
\subsubsection*{Elmendorf's construction}
With enriched and derived $\CoEnd$ in hand, we describe
Elmendorf's construction in this language. Let $X$ be a $G$-space. Recall that the functor $\FixedPtFunctor{X}\colon \Ocal_G^{\op}\to \top$ is defined by the formula $\FixedPtFunctor{X}(O)=\map_G(O, X)$. One can consider $\Phi$ as a functor of two variables: $X$ and $O$.
A key property of $\Phi$ is that it preserves the Bousfield-Kan model for a homotopy
colimit in the variable~$X$~\cite[Chapter XII, \S 5]{Bousfield-Kan}.

\begin{lemma}\label{lemma: commute hocolim}
Suppose that $\{X_\alpha\}_{\alpha\in I}$ is a diagram of $G$-spaces. Let $\hocolim$ denote the Bousfield-Kan homotopy colimit, and set
$X=\underset{\alpha\in I}{\hocolim} X_{\alpha}$.
There is a natural isomorphism of functors
\[
\underset{\alpha\in I}{\hocolim} \FixedPtFunctor{X_{\alpha}}
\to
\FixedPtFunctor{X}.
\]
\end{lemma}

\begin{proof}
We need to show that for every $O\in \Ocal_{G}$, the following map is a homeomorphism:
\[
\underset{\alpha\in I}{\hocolim} \left(\Largestrut
            \map_G\left( O, X_{\alpha}\right)
                            \right)
\longrightarrow
\map_G\left(O, \underset{\alpha\in I}{\hocolim} X_{\alpha}\right).
\]
Representing $O$ as $G/H$ for some subgroup $H$ of~$G$, we can identify $\map_G(O, X)$ with~$X^H$.
The lemma follows because fixed points of an action of a compact Lie group commute with Bousfield-Kan homotopy colimits \cite[Proposition~B.1(iv)]{Schwede-Global}.
\end{proof}

The functor $\FixedPtFunctor{X}$ is an enriched
functor $\Ocal_G^{\op}\to \top$,
and can be paired with an enriched functor
$\Ocal_G\to \top$ in a derived \CoEnd.
Let $\specialorbit$ be a particular object of~$\Ocal_G$, i.e., a $G$-orbit, and consider the representable functor $F(\whatever)=\map_G(\specialorbit, \whatever)$.
By the coYoneda Lemma
(Lemma~\ref{lemma: coyoneda}), there is a natural equivalence
$F\times^h_{\orbit} \FixedPtFunctor{X}\rightarrow F(X)$, or
\[
\left[\Largestrut\map_G(\specialorbit, \orbit)\right]
     \times^h_{\orbit\in\Ocal_{G}}
     \left[\Largestrut\map_{G}(\orbit,X)\right]
\xrightarrow{\ \simeq\ }
\map_{G}(\specialorbit,X).
\]
In particular,
taking $Q$ to be the free $G$-orbit, $\specialorbit=G$, we have a natural isomorphism $\FixedPtFunctor{X}(G)\cong X$, giving an assembly map
\begin{equation}   \label{eq: assembly}
\orbit\  \times^h_{\orbit\in\Ocal_{G}} \,\FixedPtFunctor{X}(\orbit)\longrightarrow X.
\end{equation}
Note that every orbit $O$ has a left action of $G$, and all the face and degeneracy maps in the simplicial object defining $\orbit \times^h_{\orbit} \FixedPtFunctor{X}(\orbit)$ are $G$-equivariant. It follows that $\orbit \times^h_{\orbit} \FixedPtFunctor{X}(\orbit)$ has a natural action of $G$ via the action on~$\orbit$. The assembly map is $G$-equivariant by inspection.
Elmendorf's theorem says that it is actually a weak $G$-equivalence.

\begin{theorem}\cite[Theorem 1]{Elmendorf}
      \label{theorem: elmendorf}
The assembly map $\orbit\, \times^h_{\orbit\in\Ocal_{G}} \FixedPtFunctor{X}(\orbit)\longrightarrow X$ is a weak $G$-equivalence.
\end{theorem}

\medskip
\subsubsection*{Approximation relative to a collection of subgroups}
A metastatement of Elmendorf's theorem is that $G$-spaces are essentially the same thing as continuous functors from $\Ocal_G^{\op}$ to~$\top$.
Let $\Ccal$ be a collection of subgroups of~$G$, i.e., a set of closed subgroups that is closed under conjugation (but not necessarily under passage to subgroups). Let $\Ocal_\Ccal$ be the corresponding orbit category, the full subcategory of~$\Ocal_{G}$ consisting of orbits whose isotropy groups are in~$\Ccal$.
The metastatement of Elmendorf's theorem can be generalized as follows:
$G$-spaces whose isotropy groups are contained in $\Ccal$ are essentially the same thing as
continuous functors from $\Ocal_\Ccal^{\op}$ to~$\top$.
For example, see~\cite{Stephan}, where a version of this statement is established for $G$-objects in a general model category, although it appears that $\Ccal$ is always assumed to contain the trivial subgroup.

We now make precise the version of this idea that we need (see in particular Corollary~\ref{corollary: detect} below). First, a definition. Recall that if $Z$ is a $G$-space, then $\Iso(Z)$ denotes the collection of isotropy subgroups of~$Z$.

\begin{definition}   \label{definition: approximation}
Let $X$ be a $G$-space, and $\Ccal$ a collection of subgroups of $G$. A $G$-map $Z\rightarrow X$ is called a \defining{$\Ccal$-approximation} of $X$ if it satisfies the following properties.
\begin{enumerate}
\item $\Iso(Z)\subseteq \Ccal$.
\item If $H\in \Ccal$, then the map $Z^H \to X^H$ is a weak equivalence.
\end{enumerate}
\end{definition}

\begin{remark*}
When $G$ is a finite group, approximation with respect to a collection of subgroups of~$G$ was discussed in detail in~\cite{Arone-Dwyer}. The construction and its properties extend to compact Lie groups without any difficulty, but we do not know if it is documented in the literature in the generality we need.
\end{remark*}

$\Ccal$-approximations have the usual existence and uniqueness statements. We will first give a construction and then prove that it has the desired properties.

\begin{definition}
\label{definition: C-approximation construction}
Let $X_\Ccal$ be the homotopy \CoEnd
\[
X_\Ccal \definedas \orbit \times^h_{\orbit\in \Ocal_\Ccal} \FixedPtFunctor{X}(\orbit).
\]
\end{definition}

Notice that the inclusion of orbit categories $\Ocal_\Ccal\hookrightarrow \Ocal_G$ induces a map
\[
\orbit \times^h_{\orbit\in \Ocal_\Ccal} \FixedPtFunctor{X}(\orbit)
\longrightarrow
\orbit \times^h_{\orbit\in \Ocal_G} \FixedPtFunctor{X}(\orbit),
\]
and therefore, by~\eqref{eq: assembly},
a map of $G$-spaces $\alpha_{\footnotestrut X}: X_\Ccal\longrightarrow X$. It is clear from
Definition~\ref{definition: C-approximation construction} that $X_\Ccal$ depends functorially on $X$, and that the map $\alpha_{X}\colon X_\Ccal\to X$ is a natural transformation.

\begin{proposition}\label{proposition: C-Elmendorf}
\PropositionCElmendorfText
\end{proposition}

\begin{proof}
Suppose first that $X\in \Ocal_\Ccal$.
In this case the functor $\FixedPtFunctor{X}(O)=\map_G(O, X)$ is representable as a functor from $\Ocal_\Ccal^{\op}$ to $\top$.
Using the coYoneda lemma (Lemma~\ref{lemma: coyoneda}), it follows that the map
\[
X_\Ccal\,\definedas\,\orbit \times^h_{\orbit\in \Ocal_\Ccal} \FixedPtFunctor{X}(\orbit)
\longrightarrow X
\]
is a weak $G$-equivalence.

On the other hand, it follows from
Lemma~\ref{lemma: commute hocolim} that the functor $X\mapsto X_{\Ccal}$ preserves homotopy colimits. Therefore the class of spaces~$X$ for which the map $\alpha_{\footnotestrut X}:X_{\Ccal}\longrightarrow X$ is a weak $G$-equivalence is closed under homotopy colimits. Hence this class contains all $G$-CW-complexes $X$ for which $\Iso(X)\subseteq \Ccal$,  establishing~\eqref{item: equivalence}.
Further, from the construction of $X_\Ccal$ we see that $\Iso(X_\Ccal)\subseteq \Ccal$, establishing~\eqref{item: isotropy groups}.

To establish \eqref{item: fixed point equivalence}, the fixed-point property,
suppose that $H\in \Ccal$. The map $\left(X_\Ccal\right)^H\to X^H$ induced by $\alpha_{\footnotestrut X}$ can be identified with the map
\[
\map_G(G/H, \orbit) \times^h_{\orbit\in \Ocal_\Ccal} \FixedPtFunctor{X}(\orbit)
  \longrightarrow \FixedPtFunctor{X}(G/H)=\map_G(G/H, X).
\]
Again by the coYoneda lemma (Lemma~\ref{lemma: coyoneda}), this map is an equivalence because the functor $\map_G(G/H, \whatever)$ is representable.

Finally, to establish uniqueness, suppose that both $\alpha_{X}$ and $f: Y\to X$ satisfy
properties \eqref{item: isotropy groups} and~\eqref{item: fixed point equivalence}, i.e. $\Iso(Y)\subseteq \Ccal$ and
$Y^H \to X^H$ is a weak equivalence for $H\in\Ccal$.
We would like to show that $Y$ is naturally $G$-equivalent to~$X_{\Ccal}$. Consider the diagram
\[
\begin{CD}
Y_\Ccal @>>> X_\Ccal \\
@V{\simeq}V{\alpha_{\tinystrut Y}}V @VV{\alpha_{\tinystrut X}}V\\
Y @>{f}>> X.
\end{CD}
\]
We do not know anything about the right vertical map~$\alpha_{X}$, since we have made no assumptions on the isotropy groups of~$X$. However,
the left vertical map is a weak $G$-equivalence because of
the first item of the proposition, since we assumed $\Iso(Y)\subseteq \Ccal$.
The top horizontal map is also a weak $G$-equivalence because we assumed
$Y^H \to X^H$ is a weak equivalence for all $H\in\Ccal$, and so the map from $Y_{\Ccal}$ to $X_{\Ccal}$ is the geometric realization of a map between simplicial spaces that is a level-wise weak $G$-equivalence.
Therefore $Y$ is canonically weakly $G$-equivalent
to~$X_\Ccal$.
\end{proof}

\begin{corollary}\label{corollary: detect}
Let $X$ and $Y$ be $G$-CW-complexes, and suppose that $\Iso(X)\cup \Iso(Y)\subseteq \Ccal$. Suppose that $f\colon X\to Y$ is a $G$-map, and that $f^H\colon X^H \to Y^H$ is a weak equivalence for all $H\in \Ccal$. Then $f$ is a weak $G$-equivalence.
\end{corollary}

\begin{proof}
Consider the commutative square
\[
\begin{CD}
X_\Ccal @>>> Y_\Ccal \\
@VVV @VVV\\
X @>>> Y.
\end{CD}
\]
The vertical maps are weak $G$-equivalences by Proposition~\ref{proposition: C-Elmendorf}. The top horizontal map is a map between homotopy {\CoEnd}s
\[
\orbit \times^h_{\orbit\in\Ocal_\Ccal} \FixedPtFunctor{X}(\orbit)
\longrightarrow
\orbit \times^h_{\orbit\in\Ocal_\Ccal} \FixedPtFunctor{Y}(\orbit).
\]
By assumption, the map $ \FixedPtFunctor{X}(\orbit)\to  \FixedPtFunctor{Y}(\orbit)$ is a weak equivalence for every $\orbit\in\Ocal_\Ccal$. It follows that the top horizontal map (of homotopy {\CoEnd}s) is a weak $G$-equivalence, and hence the lower horizontal map is as well.
\end{proof}

\section{Bredon (co)homology of spaces}
\label{section: Bredon homology}

In this section we review some standard material about Bredon (co)homology
for a compact Lie group.
The most important result is a criterion for a $G$-equivariant map to induce an isomorphism in Bredon (co)homology (Lemma~\ref{lemma: fixed points iso implies bredon}).
The criterion is certainly known, but we are not sure if it is recorded in the literature in the generality that we need.

As before, let $G$ be a compact Lie group, and let $\Ocal_G$ denote its orbit category,
whose morphism sets are naturally topologized as subspaces of orbits of~$G$.
Let $\pi_0\Ocal_G$ be the homotopy category of~$\Ocal_G$,
that is, the category obtained by taking path components of the morphism spaces.

\begin{definition}    \label{definition: coefficient system}
Let $\Acal$ be an
abelian\footnote{We do not actually need the full generality of using an abelian category as the target of a coefficient system, but it is sometimes convenient to have $M$
at least take values in graded abelian groups rather than just abelian groups.}
category. A \defining{coefficient system}
for $G$ with values in~$\Acal$
is a functor $M:\Ocal_{G}\rightarrow\Acal$ that factors through the projection functor $\Ocal_{G}\rightarrow\pi_{0}\Ocal_{G}$.
We say $M$ is a \defining{homological coefficient system} if $M$ is covariant,
and a \defining{cohomological coefficient system} if $M$ is contravariant.
\end{definition}

We outline the agenda for this section, with details to follow. We can identify $\Acal$ with the subcategory of chain complexes over $\Acal$ (denoted $\ch_\Acal$) consisting of chain complexes concentrated in dimension~$0$. With this identification we may consider a coefficient system $M$ as a functor from $\Ocal_G$ to~$\ch_{\Acal}$.
Let $\top_G$ be the category of  $G$-spaces. Consider the
derived Kan extension of the functor $M\colon \Ocal_G\to \ch_{\Acal}$ along the inclusion $\Ocal_G\hookrightarrow \top_G$. This is a functor from $\top_G$ to $\ch_{\Acal}$. Because $M$ factors through $\pi_{0}\Ocal_{G}$, it follows that the derived Kan extension is a {\em homotopy} functor from $\top_G$ to~$\ch_\Acal$, in the sense that it takes homotopy equivalences to quasi-isomorphisms. The homology of this functor defines Bredon
(co)homology with coefficients in~$M$, according to the variance of~$M$. We then use this definition to prove our criterion for Bredon (co)homology isomorphism.
The rest of this section gives details.

\medskip
\subsubsection*{Definition of Bredon (co)homology}
Let $\ch_{\integers}$ denote the category of chain complexes of abelian groups, and let
$\nc:\top \to \ch_{\integers}$ be the normalized singular chains functor,
which is a lax symmetric monoidal functor.
Suppose $\Ocal$ (shortly to be the orbit category~$\Ocal_{G}$)
is a category enriched over $\top$.
We define $\nc\Ocal$ as the category obtained by applying the normalized
singular chains functor $\nc(\whatever;\integers)$ to the hom objects of~$\Ocal$.
The category $\nc\Ocal$ has the same objects as~$\Ocal$, and it is enriched over~$\ch_{\integers}$ (in other words, $\nc\Ocal$ is a dg-category). Furthermore, suppose $F\colon \Ocal\to \top$ is a topologically enriched functor. Then the functor
$\nc F\colon \nc\Ocal \to \ch_{\integers}$ is a functor enriched over~$\ch_{\integers}$.

Now let $X$ be a $G$-space, and take $F$ to be the fixed point functor $\FixedPtFunctor{X}\colon \Ocal_G^{\op}\to \top$. After composing with~$\nc$, we obtain the functor $\nc\FixedPtFunctor{X}\colon\nc\Ocal_G^{\op}\to \ch_{\integers}$.
Let $M$ be a (co)homological
coefficient system, which one may think of as a (possibly contravariant) enriched functor $M\colon\nc\Ocal_G\to \ch_{\integers}$.

\begin{definition}\label{definition: bredon}
\hfill
\begin{enumerate}
\item    \label{item: Bredon homology}
If $M$ is a homological coefficient system,
the \defining{Bredon homology of $X$ with coefficients in~$M$}, denoted $\HH^{G}_*(X;M)$, is the homology of the following enriched homotopy coend
(Definition~\ref{defn: End and CoEnd}\eqref{item: der coend}) of chain-complex-valued functors:
\[
M  \htensor_{\nc \Ocal_G} \nc \FixedPtFunctor{X}.
\]

\item
If $M$ is a cohomological coefficient system, then the
\defining{Bredon cohomology of $X$ with coefficients in~$M$},
denoted $\HH_{G}^*(X;M)$, is the homology of the following enriched homotopy end
(Definition~\ref{defn: End and CoEnd}\eqref{item: der end})
of chain-complex-valued functors:
\[
\hnat_{\nc \Ocal_G}(\nc\FixedPtFunctor{X}, M).
\]
\end{enumerate}
\end{definition}

The Bredon (co)homology groups of Definition~\ref{definition: bredon}
satisfy an equivariant version of the Eilenberg-Steenrod axioms.
In particular, Lemma~\ref{lemma: dimension} below says that Bredon (co)homology satisfies an equivariant version of the dimension axiom. All the other Eilenberg-Steenrod axioms are
 exactly the same for equivariant (co)homology as their non-equivariant counterparts~\cite{Willson}.

\begin{lemma}\label{lemma: dimension}
 If $X\in \Ocal_G$ then there is an isomorphism
 $\HH^{G}_0(X;M)\cong M(X)$, natural in $X$. Moreover, $\HH^{G}_i(X;M)=0$ for $i>0$.
 The dual statement holds for cohomology.
\end{lemma}

\begin{proof}
If $X\in \Ocal_G$, then $ \nc \FixedPtFunctor{X}$ is a representable contravariant functor, so the
claim follows from Lemma~\ref{lemma: coyoneda}.
\end{proof}

\begin{remark}\label{remark: cellular}
An alternative, commonly used construction of Bredon homology uses cellular chains. For example, see~\cite{Willson} or~\cite[Chapter I.4]{May-Alaska}. Any two theories that satisfy the Eilenberg-Steenrod axioms and agree on $\Ocal_G$ are naturally isomorphic on the category of $G$-CW-complexes~\cite[Corollary 3.2]{Willson}. Therefore the cellular definition agrees with Definition~\ref{definition: bredon}. We chose to use homotopy $\CoEnd$ and $\End$ in our definition, because some invariance results seem most apparent from this perspective. In particular, see Lemma~\ref{lemma: fixed points iso implies bredon} below.
\end{remark}

\medskip
\subsubsection*{Spaces with isotropy in a collection}
As in Section~\ref{section: preliminaries}, suppose $\Ccal$ is a collection of subgroups of~$G$ and $\Ocal_{\Ccal}$ is the corresponding subcategory of the orbit category~$\Ocal_{G}$.
We will now see another manifestation of the metastatement that if $\Ccal$ is a collection of subgroups,
then spaces whose isotropy is contained in $\Ccal$ are essentially the same as topologically enriched
functors from $\Ocal_\Ccal^{\op}$ to~$\top$.
The following lemma is closely analogous to
Proposition~\ref{proposition: C-Elmendorf}\eqref{item: equivalence}.

\begin{lemma}      \label{lemma: sufficient}
Suppose that $X$ is a $G$-CW-complex all of whose isotropy groups are in~$\Ccal$.
Then the inclusion of subcategories $\Ocal_\Ccal\to \Ocal_G$ induces a quasi-isomorphism of chain complexes
\begin{equation}   \label{eq: quasi-iso}
M \htensor_{\nc \Ocal_\Ccal}  \nc \FixedPtFunctor{X}
\longrightarrow
M \htensor_{\nc \Ocal_G} \nc \FixedPtFunctor{X}.
\end{equation}
In particular, either chain complex can be used to
calculate $\HH^{G}_*(X;M)$. A dual statement holds for cohomology.
\end{lemma}

\begin{proof}
If $X$ is an object of~$\Ocal_\Ccal$, the result follows by applying the (co)Yoneda lemma
(Lemma~\ref{lemma: coyoneda}) to both sides
of~\eqref{eq: quasi-iso}.
However, the functor $\FixedPtFunctor{X}$ preserves homotopy colimits in the variable~$X$ (Lemma~\ref{lemma: commute hocolim}). Therefore the class
 of spaces $X$ for which the lemma holds is closed under arbitrary homotopy colimits, and in particular it contains the class of spaces whose isotropy is contained in~$\Ccal$.
\end{proof}

\medskip
\subsubsection*{Criterion for isomorphic Bredon (co)homology}
For our applications we need to determine circumstances under which homological properties
of the fixed point diagram $\FixedPtFunctor{X}$ determine Bredon (co)homology.
The following lemma is an algebraic analogue of Corollary~\ref{corollary: detect}.
We will use the lemma later to prove that if a $G$-map induces a mod~$p$ homology isomorphism on certain fixed point sets, then it actually induces an isomorphism on Bredon homology for coefficients that are \pdash local in a suitable sense (Proposition~\ref{proposition: isomorphism}).

\begin{lemma} \label{lemma: fixed points iso implies bredon}
Let $f:X\longrightarrow Y$ be a map of $G$-spaces, and let $M$ be a coefficient system for Bredon (co)homology.
Assume that for all isotropy groups
$H_1, H_2\in\Iso(X)\cup\Iso(Y)$, the map
$f:X^{H_1}\longrightarrow Y^{H_{1}}$
induces an isomorphism of ordinary (co)homology groups with coefficients in~$M(G/H_2)$.
Then $f$ induces an isomorphism on Bredon (co)homology with coefficients in~$M$.
\end{lemma}

\begin{proof}
Consider the case of homology. Let $\Ccal=\Iso(X)\cup\Iso(Y)$. The inclusion
$\Ocal_{\Ccal}\rightarrow\Ocal_{G}$ gives us a commutative square
\begin{equation}    \label{diag: desired iso}
\begin{CD}
M \htensor_{\nc \Ocal_{\Ccal}} \nc \FixedPtFunctor{X}
     @>>> M \htensor_{\nc \Ocal_{G}} \nc \FixedPtFunctor{X}\\
@VVV @VVV\\
M \htensor_{\nc \Ocal_{\Ccal}} \nc \FixedPtFunctor{Y}
     @>>> M \htensor_{\nc \Ocal_{G}}  \nc \FixedPtFunctor{Y}.
\end{CD}
\end{equation}
Our goal is to show that the right vertical
map is a quasi-isomorphism. The horizontal maps are quasi-isomorphisms
by Lemma~\ref{lemma: sufficient}, so it is enough to prove that the left vertical map is a quasi-isomorphism.
By definition, the domain, i.e. the homotopy coend $M \htensor_{\nc \Ocal_{\Ccal}} \nc \FixedPtFunctor{X}$,
is the total complex of a simplicial chain complex (Remark~\ref{remark: realization}). Since the left
vertical map is induced by a homomorphism between simplicial chain complexes, it is enough to prove that it is a quasi-isomorphism in simplicial degree $n$
for all~$n$.

Suppose that $H_{0},\ldots,H_{n}$ are subgroups of~$G$, and define the chain complex
\[
\Zcal\left(H_{0},\ldots,H_{n}\right)
\definedas
\nc\left((G/H_1)^{H_0}\right)
\otimes\cdots \otimes
\nc\left((G/H_n)^{H_{n-1}}\right).
\]
The chain complex $\Zcal\left(H_{0},\ldots,H_{n}\right)$ could be zero, if some $H_{i}$ were not conjugate to a subgroup of~$H_{i+1}$, but it is certainly a complex of free
abelian groups, by definition of~$\nc$.
In simplicial degree $n$ of the left vertical map of~\eqref{diag: desired iso},
we have the homomorphism
\begin{equation}  \label{eq: degree n}
\begin{CD}
\bigoplus_{H_0, \ldots, H_n\in \Ccal}\
\left[M(G/H_0)
\otimes \Zcal\left(H_{0},...,H_{n}\right)
\otimes \nc(X^{H_n})\right]\\
@VVV\\
\bigoplus_{H_0, \ldots, H_n\in \Ccal} \
\left[M(G/H_0)
\otimes \Zcal\left(H_{0},...,H_{n}\right)
\otimes \nc(Y^{H_n})\right].
\end{CD}
\end{equation}
Our assumption implies that the homomorphism
\[
M\left(G/H_0\right)\otimes \nc\left(X^{H_n}\right)
\ \longrightarrow\
M\left(G/H_0\right)\otimes \nc\left(Y^{H_n}\right)
\]
is a quasi-isomorphism. It follows that~\eqref{eq: degree n} is also a quasi-isomorphism,
because $\Zcal\left(H_{0},...,H_{n}\right)$ is free.
Therefore the left vertical map in~\eqref{diag: desired iso}, an induced map of homotopy {\CoEnd}s, is a quasi-isomorphism as well.

The proof in the cohomology case is similar.
\end{proof}


\section{Reduction to a maximal \pdash toral subgroup}
\label{section: reduction to p-Sylow}

Given a compact Lie group~$G$, let $\all{G}$ denote the collection of all \pdash toral subgroups of~$G$. A~principal goal of this paper is to establish that if $X$ is a $G$-space, then the approximation~$X_{\all{G}}$ is good enough to compute the Bredon (co)homology of $X$. That is, under certain assumptions on a coefficient system~$M$, the approximation map $X_{\all{G}}\to X$ induces an isomorphism
\begin{equation}   \label{eq: desired iso}
\HH^{G}_*(X_{\all{G}};M)\xrightarrow{\cong} \HH^{G}_*(X;M)
\end{equation}
on Bredon homology, and a similar statement for cohomology.
An important part of our approach is a formulation of \eqref{eq: desired iso} as a property of the pair~$G$ and~$M$, that is, a property for all $G$-spaces at once, rather than focusing on just one space at a time.

\begin{definition}     \label{definition: approximation property}
Let $G$ be a compact Lie group and let $M$ be a Bredon coefficient
system for~$G$. We say that the group $G$ has the
\defining{\pdash toral approximation property for $M$}
if, for all $G$-CW-complexes~$X$,
the map $X_{\all{G}}\rightarrow X$ is an $M$-(co)homology isomorphism.
\end{definition}

In Section~\ref{section: toral}, we will prove that
all compact Lie groups have the \pdash toral approximation property for suitable coefficient systems. (See Theorem~\ref{thm: approximation theorem} and Proposition~\ref{proposition: isomorphism}.) The goal in this section is preparation for
Section~\ref{section: toral} in two ways. First, we
explain what ``suitable" means (the ``\upanddown/ condition for~$p$," Definition~\ref{definition: upanddown}).  Second, we reduce from a compact Lie group $G$ to
a maximal \pdash toral subgroup, as stated in the following proposition.

\begin{proposition} \label{proposition: p-Sylow subgroup enough}
  Let $G$ be a compact Lie group with maximal \pdash toral subgroup~$P$, and let $M$ be a (co)Mackey functor for $G$ that satisfies the \upanddown/
  condition for~$p$ (Definition~\ref{definition: upanddown}).  If $P$ has the \pdash toral approximation property
  for $M\restrictedto{P}$, then $G$
  has the \pdash toral approximation property for~$M$.
\end{proposition}

We begin with a discussion of (co)Mackey functors, so that we can state the
required condition on the coefficients in
Definition~\ref{definition: upanddown}.
The proof of Proposition~\ref{proposition: p-Sylow subgroup enough} is at the end of the section.

\subsubsection*{(co)Mackey functors and compact Lie groups}
In Section~\ref{section: Bredon homology}, we recalled the Bredon
(co)homology of $G$-spaces
with coefficients in a functor from $\Ocal_{G}$ to an abelian
category (a ``coefficient system").
We now focus on coefficient systems that have an additional structure,
that of a (co)Mackey functor
(Definition~\ref{definition: Mackey and coMackey}). The new tool that becomes available as a consequence of this assumption is the transfer homomorphism.


Let $\SO{G}$ be the stable homotopy category of $G$-orbits. Thus objects of $\SO{G}$ are equivariant suspension spectra~$\Sigma^\infty O_+$, where $O$ is a $G$-orbit, and morphisms are stable $G$-equivariant homotopy classes of maps between such spectra. Note that $\SO{G}$ is an additive category.

\begin{definition}    \label{definition: Mackey and coMackey}
Let $\Acal$ be an abelian category.
A \defining{coMackey functor} for $G$ with values in $\Acal$ is an additive functor
$M\colon \SO{G}\to \Acal$, while a \defining{Mackey functor} for $G$ with values in $\Acal$ is
an additive functor $M\colon \SO{G}^{\op}\to \Acal$.
\end{definition}
\noindent(The convention that Mackey functors are contravariant and coMackey functors are covariant is established in the literature, and we are not going to try to subvert it.)

\begin{remark*}
When $G$ is a finite group, Spanier-Whitehead duality induces an isomorphism between the categories $\SO{G}$ and $\SO{G}^{\op}$ that is the identity on objects.
This in turn induces
an isomorphism between the categories of Mackey and coMackey functors that is the identity
on objects. But for compact Lie groups that are not finite groups,
this duality does not exist: the Spanier-Whitehead dual of $G/H$ is not generally equivalent to~$G/H$, or to any other object of~$\SO{G}$.
This is why we must deal explicitly with both Mackey and coMackey functors in this work, where in \cite{ADL2} we were able to suppress the distinction.
\end{remark*}

\subsubsection*{Extension to $G$-spectra}

Note that the assignment $O\mapsto \Sigma^\infty O_+$ defines a functor $\Ocal_G\to \SO{G}$.
Thus a coMackey functor $M$ (in the sense of Definition~\ref{definition: Mackey and coMackey})
gives rise to an ``ordinary'' homological coefficient system
(in the sense of Definition~\ref{definition: coefficient system}). We generally denote the coefficient system defined by a coMackey functor $M$ with the same letter~$M$. Dually, a Mackey functor gives rise to a cohomological coefficient system.

The key point is that if a coefficient system (defined on~$\Ocal_{G}$) comes from restricting a (co)Mackey functor, defined on~$\SO{G}$, then one can extend the associated (co)homology theory from the category of $G$-spaces to the category of $G$-spectra. As a consequence, (co)homology with coefficients in a (co)Mackey functor is equipped with transfer homomorphisms, as discussed below.

The constructions of (co)homology groups of $G$-spectra with coefficients in a
(co)Mackey functor are very similar to the constructions of Bredon homology and cohomology alluded to in Remark~\ref{remark: cellular}.
Namely, when $X$ is a $G$-CW spectrum, the cellular chain complex of~$X$ has a natural structure of a chain complex of Mackey functors. More explicitly, suppose $X^n$ is the $n$-dimensional skeleton of $X$. Then the $n$-th cellular chains on $X$, which we denote by $C^{\cw}_n(X)$, is the Mackey functor defined by the formula
\[
C^{\cw}_n(X)(\Sigma^\infty O_+)
   \definedas
   \left[S^n \wedge \Sigma^\infty  O_+\, , X^n/X^{n-1}\right]_G,
\]
where $[\whatever,\whatever]_G$ denotes the group of homotopy classes of maps betwen $G$-spectra.
One then defines the homology groups of $X$ with coefficients in a coMackey functor~$M$ as the homology of the homotopy coend
\begin{equation}     \label{eq: construction}
M \otimes^h_{\SO{G}} C^{\cw}_*(X).
\end{equation}
Dually, if $M$ is a Mackey functor, the cohomology groups of $X$ with coefficients in~$M$ are the homology groups of the homotopy end
\[
\hnat_{\SO{G}}(C^{\cw}_*(X), M).
\]

Construction~\eqref{eq: construction} and its cohomological analogue are the analogue for $G$-spectra of
Definition~\ref{definition: bredon} 
for $G$-spaces. Standard cellular techniques guarantee that the construction is homotopy invariant. For more details, see~\cite[Chapter XIII, Section 4]{May-Alaska} or~\cite{Lewis-May-McClure}.
If $X$ is a $G$-space and $M$ is a coMackey functor, then the Bredon homology groups $\HH_*^G(X; M)$ of the space~$X$, with $M$ considered as a coefficient system, are canonically isomorphic to the new homology groups
$\HH_*^G\left(\Sigma^\infty X_+; M\right)$ of the suspension spectrum of $X$ with coefficients in $M$ as a coMackey functor~\cite[page 138]{May-Alaska}.

\subsubsection*{Transfers}
(Co)homology theories on the category of $G$-spectra are
sometimes called $RO(G)$-graded theories.
For the purpose of this paper, the most important consequence of a (co)homology theory being $RO(G)$-graded is that it has transfers. In particular, suppose $\eta\colon E\to B$ is a $G$-equivariant fibration whose fiber is a finite complex~$F$.
The equivariant Becker-Gottlieb transfer of $\eta$ is a $G$-map $\tr\colon \Sigma^\infty B_+ \to \Sigma^\infty E_+$. Therefore for any coMackey functor~$M$, such a fibration $\eta$ gives rise to a transfer homomorphism $\HH_*^G(B; M)\to \HH_*^G(E; M)$. Dually, if $M$ is a Mackey functor, then a fibration $\eta$ gives rise to a transfer homomorphism  $\HH^*_G(E; M)\to \HH^*_G(B; M)$.

\subsubsection*{Restriction and induction}
Suppose $G$ is a group and $H$ a subgroup of $G$.
There is an induction functor $i\colon \Ocal_H\to \Ocal_G$ given by
$O\mapsto G\times_H O$. If $M$ is a coefficient system for $G$,
the composition $Mi$ is a cofficient system for~$H$. We call $Mi$ the
\defining{restriction of $M$ to $H$}, and denote it by $M\restrictedto{H}$. The following is a standard result, a consequence of the fact that (derived) left Kan exension is (derived) left adjoint to restriction.

\begin{lemma}\label{lemma: induction}
With $G$ and $H$ as above, let $X$ be a space with an action of $H$ and let $M$ be a coefficient system for $G$. There is a natural isomorphism
\[
\HH_*^H\left(X; M\restrictedto{H}\right)\cong \HH_*^G\left(\strut G\times_H X; M\right),
\]
where if $M$ is cohomological then homology is replaced with cohomology.
\end{lemma}
Note also that if $X$ is a space with an action of $G$, there is a canonical homeomorphism of $G$-spaces $G/H \times X\cong G\times_H X$. Therefore in this case we have an isomorphism
$\HH_*^H(X; M\restrictedto{H})\cong \HH_*^G(G/H \times X; M)$.

\subsubsection*{The \upanddown/ condition for~$p$}
\begin{definition}   \label{definition: upanddown}
  We say that a (co)Mackey functor $M$ for $G$
  satisfies the \defining{\upanddown/ condition for the prime $p$} if, for all
  subgroups $H\subseteq G$ and a maximal \pdash toral subgroup $P$ of $G$, the
  composite map on Bredon (co)homology is an isomorphism:
\begin{align*}
 \HH^{G}_*(G/H;M)\xrightarrow{\ \tr\ }  &\HH^{G}_*(G/P\times G/H;M)
      \longrightarrow  \HH^{G}_*(G/H;M)\\
 \HH_{G}^*(G/H;M) \longrightarrow       &\HH_{G}^*(G/P\times G/H;M)
      \xrightarrow{\ \tr\ }  \HH_{G}^*(G/H;M).
\end{align*}
\end{definition}

\begin{lemma}\label{lemma: p-transfer}
Suppose that a (co)Mackey functor $M$ satisfies the \upanddown/ condition for the prime~$p$. Then for every $G$-space~$X$, the composite map on Bredon (co)homology is an isomorphism:
\begin{align*}
 \HH^{G}_*(X;M)\xrightarrow{\ \tr\ }  &\HH^{G}_*(G/P\times X;M)
      \longrightarrow  \HH^{G}_*(X;M)\\
 \HH_{G}^*(X;M) \longrightarrow       &\HH_{G}^*(G/P\times X;M)
      \xrightarrow{\ \tr\ }  \HH_{G}^*(X;M).
\end{align*}
\end{lemma}

\begin{proof}
For concreteness, we work with coMackey functors and homology, the dual case being treated in the same way.
Since $M$ is a coMackey functor, homology with coefficients in $M$ can be defined for $G$-spectra. Given a $G$-space~$X$, we have
$\HH^{G}_*(X;M)\cong  \HH^{G}_*(\Sigma^\infty X_+;M)$, so we can work with
$G$-spectra rather than $G$-spaces to establish the lemma.

We claim that the lemma holds for all $G$-spectra. That is, we claim that for any $G$-spectrum $E$, the following composite is an isomorphism:
\[
 \HH^{G}_*(E;M)\xrightarrow{\ \tr\ }
      \HH^{G}_*(\Sigma^\infty G/P_+\wedge E;M)
      \longrightarrow  \HH^{G}_*(E;M).
\]
By assumption, the lemma holds when $E$ is of the form $E=\Sigma^\infty G/H_+$
and hence also when $E$ is some suspension of such a spectrum.
Since homology takes wedge sums to direct sums, the lemma also holds when $E$ is a wedge sum of spectra that are suspensions of~$\Sigma^\infty G/H_+$.

Induction over cells, together with Mayer-Vietoris, establish the lemma when $E$ is a finite-dimensional CW-spectrum. Finally, every spectrum is a filtered homotopy colimit of finite CW-spectra, and homology commutes with filtered homotopy colimits, so the claim holds for all $G$-spectra $E$. In particular, it holds when $E=\Sigma^\infty X_+$ for a $G$-space $X$, which is the case that is stated in the lemma.
\end{proof}

The point of Lemma~\ref{lemma: p-transfer} is that $\HH^{G}_*(G/P\times X;M)\cong \HH^{P}_*(X;M|_P)$ may be easier to analyze than $\HH^{G}_*(X;M)$ by virtue of $P$ being \pdash toral. The \upanddown/ condition guarantees that $\HH^{G}_*(X;M)$ is a direct summand of $\HH^{G}_*(G/P\times X;M)$ (see~\eqref{eq: transfer diagram}). Definition~\ref{definition: upanddown}
replaces our assumption in the finite-group context of \cite{ADL2} that the coefficients were projective relative to the family of \pdash subgroups of a finite group.

\begin{proof}[Proof of Proposition~\ref{proposition: p-Sylow subgroup enough}]
We will consider the case when $M$ is a coMackey functor and focus on homology. The proof of the other case is essentially the same.

Consider the diagram
\begin{equation}   \label{eq: transfer diagram}
\begin{CD}
G\times_{P} X_{\all{G}}@>>> G \times_{P} X\\
@VVV @VVV\\
X_{\all{G}}@>>> X.
\end{CD}
\end{equation}
Because $X_{\all{G}}$ and $X$ are $G$-spaces, and not just $P$-spaces, we know that
the top row is homeomorphic to
$G/P \times X_{\all{G}} \longrightarrow G/P \times X$.
The coefficients $M$ are assumed to satisfy the \upanddown/ condition for~$X$, so we know by Lemma~\ref{lemma: p-transfer} that on $H_{*}^{G}\left(\whatever; M\right)$ the
lower map is a retract of the upper map.  Hence it is sufficient to
show that the top row induces an isomorphism
\[
H_{*}^{G}\left(G\times_{P} X_{\all{G}}; M\right) \rightarrow
H_{*}^{G}\left(G \times_{P} X; M\right).
\]
Since the spaces are induced up, this means we want
to show that
\begin{equation}    \label{eq: needed iso}
H_{*}^{P}\left(X_{\all{G}}; M\restrictedto{P}\right) \rightarrow
H_{*}^{P}\left(X;  M\restrictedto{P}\right)
\end{equation}
is an isomorphism.

If $G$ were finite, then $X_{\all{G}}\rightarrow X$
would actually be a weak $P$-equivalence, but for the compact Lie case, there is an extra step. Consider the following diagram of $P$-spaces:
\[
\begin{CD}
\left(X_{\all{G}}\right)_{\all{P}} @>>> X_{\all{P}}\\
@VVV @VVV\\
X_{\all{G}} @>>> X.
\end{CD}
\]
The top row is actually a $P$-equivalence,
by Corollary~\ref{corollary: detect}, and is therefore an $M\restrictedto{P}$-homology
isomorphism. The two vertical maps are $M\restrictedto{P}$-homology
isomorphisms because we assumed that $P$ has the \pdash toral approximation
property for $M\restrictedto{P}$.
Hence the bottom row is an $M\restrictedto{P}$-isomorphism, as required for~\eqref{eq: needed iso}, and the
proposition follows.
\end{proof}


\section{\pdash toral approximations}
\label{section: toral}

Recall that in Definition~\ref{definition: approximation property}, we said that a compact Lie group $G$ satisfies the \defining{\pdash toral approximation property for $M$} if, for $G$-CW complexes~$X$, the map
$X_{\all{G}}\rightarrow X$ induces an isomorphism in Bredon (co)homology with coefficients in~$M$.
In this section, our goal is to prove that all compact Lie groups have the \pdash toral approximation property for suitably chosen coefficients, as follows.

\begin{theorem}   \label{thm: approximation theorem}
\ApproximationTheoremText
\end{theorem}

As a consequence of Theorem~\ref{thm: approximation theorem}, we will prove the following proposition, which gives a criterion for a map to induce an isomorphism in Bredon (co)homology with coefficient systems that take values in \pdash local abelian groups.
We say that a space is \defining{homologically of finite type} if all its (ordinary, integral) homology groups are finitely generated.
For example, a finite $G$-CW-complex~$X$ is homologically of finite type, as are its fixed point spaces~$X^H$.

\begin{proposition}\label{proposition: isomorphism}
\IsomorphismPropositionText
\end{proposition}

We note that if $G$ is a \emph{finite} group, methods to establish results like
Theorem~\ref{thm: approximation theorem} are well-established in the literature.
(For example, see \cite[6.4]{Dwyer-Sharp} and \cite[Proposition~4.6]{ADL2} for similar results in the context of finite groups.)
The first step is
to reduce to a Sylow \pdash subgroup $P\subseteq G$ via a standard chain-level transfer argument.
The second step is to show that when $P$ is a \pdash group, the map $X_{\all{P}}\to X$ induces an isomorphism on Bredon (co)homology. This is easy when $G$ is finite, because $X_{\all{P}}\to X$ turns out actually to be a $P$-equivalence. The essential point is that every subgroup of a \pdash group is a \pdash group, so the set of \pdash subgroups of~$P$ is the same as the set of \emph{all}  subgroups of~$P$. As a result, $X_{\all{P}}$ and $X$ have the same fixed-point data as $P$-spaces, and are weakly $P$-equivalent.

The same two steps appear when we have a compact Lie group~$G$ instead of a finite group.
The first step, reduction to a maximal \pdash toral subgroup, was addressed in
Section~\ref{section: reduction to p-Sylow} as Proposition~\ref{proposition: p-Sylow subgroup enough}. The second step is to show that $X_{\all{P}}\to X$ induces an isomorphism on Bredon (co)homology when $P$ is \pdash toral.
The difficulty in the compact Lie case is that a subgroup of a nonfinite \pdash toral group need not be \pdash toral.
As a result, $\all{P}$ is not the set of
\emph{all} subgroups of~$P$ and the map $X_{\all{P}}\to X$ need not be a $P$-equivalence, as it was in the finite group case. (This also comes up in the last part of the proof
of Proposition~\ref{proposition: p-Sylow subgroup enough}, and is the reason that an extra step is required.)
Nevertheless, Theorem~\ref{thm: approximation theorem} states that we still get a Bredon (co)homology isomorphism, even without a $P$-equivalence of spaces.

We first recall two standard lemmas, which reduce the problem of approximating a $G$-space first to the problem of approximating orbits $G/K$ (for any isotropy group~$K$, including $K=G$), and then to approximating a point.
We then prove Theorem~\ref{thm: approximation theorem},
by inducting over its \pdash toral subgroups to show that a maximal one has the \pdash toral approximation property, and then applying
Proposition~\ref{proposition: p-Sylow subgroup enough}.
And lastly, we prove Proposition~\ref{proposition: isomorphism}.

\begin{lemma} \label{lemma: approximate orbits}
  If $\left(G/K\right)_{\all{G}}\rightarrow G/K$ induces an
  $M$-(co)homology isomorphism for all subgroups
  $K\subseteq G$, then $G$ has the \pdash toral approximation property
  for $M$-(co)homology.
\end{lemma}

\begin{proof}
We focus on the homology case. The approximation functor $X\mapsto X_{\all{G}}$ preserves homotopy colimits. The ``Bredon chains" functor $X\mapsto M  \htensor_{\nc \Ocal_G} \nc \FixedPtFunctor{X}$ also preserves homotopy colimits.
It follows that the class of spaces $X$ for which the lemma holds is closed under homotopy colimits. If that class includes all orbits of~$G$, then it includes all $G$-CW-complexes.
\end{proof}

\begin{lemma}   \label{lemma: approximate *}
Suppose $K\subseteq G$. If the map $*_{\all{K}}\rightarrow *$ is
an $M\restrictedto{K}$-(co)homology isomorphism, then
$\left(G/K\right)_{\all{G}}\rightarrow G/K$ is an $M$-(co)homology
isomophism.
\end{lemma}

\begin{proof}
Consider the following commutative diagram
 \[
\xymatrix{
 \left(G\times_{K} *_{\all{K}}\right)_{\all{G}}
     \ar[r]^{\simeq_{G}}
     \ar[d]_{\simeq_{G}}
     & \left(G\times_{K} *\right)_{\all{G}}
       \ar[d] \\
 G\times_{K} *_{\all{K}}
     \ar[r]
     & G\times_{K} *
}
 \]
The left vertical map is a $G$-equivalence
because
$\Iso(G\times_{K} *_{\all{K}})\subseteq\all{G}$
(Proposition~\ref{proposition: C-Elmendorf}). The top map is seen to be a $G$-equivalence by checking if $P\in\all{G}$, then
$G\times_{K} *_{\all{K}}  \rightarrow  G\times_{K} *$ induces an equivalence of $P$-fixed points.

Our goal is to prove that the right vertical map induces an $M$-(co)homology isomorphism, and it suffices to show that the bottom map does so.
But we have assumed that the $K$-equivariant map $*_{\all{K}}\rightarrow *$ induces an isomorphism in (co)homology with coefficients
in~$M\restrictedto{K}$. The bottom map is obtained by inducing
$*_{\all{K}}\rightarrow *$ up to~$G$, and therefore induces an isomorphism in (co)homology with coefficients in $M$ as required (Lemma~\ref{lemma: induction}).
\end{proof}

\begin{corollary}   \label{cor: points are enough}
  If for all $K\subseteq G$, the map $*_{\all{K}}\rightarrow *$ is an
  $M\restrictedto{K}$-(co)homology isomorphism, then $G$ has the \pdash toral
  approximation property for $M$-(co)homology.
\end{corollary}

Corollary~\ref{cor: points are enough} reduces
Theorem~\ref{thm: approximation theorem} from a problem of approximating
a general $G$-space to a problem of approximating a point. The last ingredient for
the proof of Theorem~\ref{thm: approximation theorem} is a notion of
the size of a \pdash toral group, to index an induction.

\begin{definition}
  Assume $P$ is a \pdash toral group. The \defining{\size\ of $P$} is
  defined as $\sz{P}=(r,c)$, where $r$ is the rank of the identity
  component of $P$ (which is a torus) and $c$ is the number of
  components of~$P$.
\end{definition}

We order \pdash toral subgroups by \size\ lexicographically. Note that if
$Q\subset P$ is a strict containment of \pdash toral subgroups, then
$\sz{Q}<\sz{P}$.

\begin{proof}[Proof of Theorem~\ref{thm: approximation theorem}]

By Proposition~\ref{proposition: p-Sylow subgroup enough}, it is sufficient to establish
that a maximal \pdash toral subgroup $P$ of $G$ has the \pdash toral approximation property
(for coefficients restricted to the subgroup). So we will prove that, in fact, every \pdash toral subgroup $P$ of~$G$ has the \pdash toral approximation property for~$M\restrictedto{P}$.

We induct over $\sz{P}$. To start the induction, suppose that $P$ is a \pdash toral group with $\sz{P}=(0,c)$ for some~$c$, i.e., $P$ is a finite
\pdash group.  For any $P$-space~$X$, the isotropy groups of $X$
are necessarily also finite \pdash groups,
i.e. $\Iso(X)\subseteq \all{P}$. By
Proposition~\ref{proposition: C-Elmendorf}, $X_{\all{P}}\rightarrow X$
is a weak $P$-equivalence and thus an $M\restrictedto{P}$-(co)homology isomorphism.

Now suppose that $\sz{P}=(r,c)$ with~$r>0$, and that any \pdash toral
subgroup $Q$ of smaller \size\ than $P$ has the \pdash toral approximation property.
By Corollary~\ref{cor: points are enough}, it suffices
to prove that for every subgroup $K\subseteq P$,
the map $*_{\all{K}}\to *$ induces an $M\restrictedto{K}$-co(homology) isomorphism. This is certainly the case when $K=P$, because
$\Iso(*)=\{P\}\subseteq\all{P}$ and so
$*_{\all{P}}\to *$ is a weak $P$-equivalence by
Proposition~\ref{proposition: C-Elmendorf}\eqref{item: equivalence}.

Now let $K\subset P$ be a proper subgroup. We need to prove that $*_{\all{K}}\to *$ is an $M\restrictedto{K}$-co(homology) isomorphism. We will
actually prove the stronger statement that $K$ has the \pdash toral approximation property for~$M\restrictedto{K}$. Let $Q$ be a maximal \pdash toral subgroup of~$K$.
Since $Q$ is a proper subgroup of~$P$, we know that
$\sz{Q}<\sz{P}$, so $Q$ has the \pdash toral approximation property for $M\restrictedto{Q}$ by the induction hypothesis. Using Proposition~\ref{proposition: p-Sylow subgroup enough} once again, this time for the pair $Q\subseteq K$, we conclude that $K$ has the \pdash toral approximation property, and we are done.
\end{proof}

\begin{proof}[Proof of Proposition~\ref{proposition: isomorphism}]
Consider the diagram of $G$-spaces
\begin{equation}   \label{diagram: X to Y}
\begin{CD}
X_{\all{G}} @>>> Y_{\all{G}}\\
@VVV @VVV\\
X @>>> Y.
\end{CD}
\end{equation}
Since $M$ satisfies the \upanddown/ condition for the prime~$p$, the vertical maps induce isomorphisms in Bredon (co)homology with coefficients in~$M$ by Theorem~\ref{thm: approximation theorem}.

On the other hand, we are set up to apply
Lemma~\ref{lemma: fixed points iso implies bredon} to show that the top horizontal map of~\eqref{diagram: X to Y} is a Bredon (co)homology isomorphism.
Consider an arbitrary subgroup $H\in\Iso(X_{\all{G}})\cup \Iso(Y_{\all{G}})$,
which is necessarily \pdash toral. Take fixed points under $H$ in diagram \eqref{diagram: X to Y} to obtain
\begin{equation}   \label{diagram: X to Y fixed points}
\begin{CD}
\left(X_{\all{G}}\right)^{H} @>>> \left(Y_{\all{G}}\right)^{H}\\
@V{\simeq}VV @V{\simeq}VV\\
X^{H} @>>> Y^{H}.
\end{CD}
\end{equation}
By assumption, the bottom horizontal map is
an ordinary mod~$p$ homology isomorphism, so the top horizontal map
is likewise. Because of the finite-type hypothesis, it follows that the top horizontal map is an isomorphism on homology with coefficients in any \pdash local group. Thus the top horizontal map is
actually an isomorphism of ordinary homology with coefficients in $M(G/K)$ for any subgroup $K\subseteq G$, because $M$ is assumed to take values in \pdash local abelian groups.
We conclude that the top horizontal map of \eqref{diagram: X to Y} induces an isomorphism in Bredon (co)homology by Lemma~\ref{lemma: fixed points iso implies bredon}. It follows that the bottom horizontal map
of \eqref{diagram: X to Y} induces an isomorphism in Bredon (co)homology also.
\end{proof}



\section{The complex of direct-sum decompositions}
     \label{section: applications}
In this section we apply our general result,
Propostion~\ref{proposition: isomorphism},
to understand the Bredon (co)homology of the space~$\Lcal_n$, the complex of direct-sum decompositions of~$\complexes^n$. The principal result is
Theorem~\ref{theorem: application}, and much of the section is devoted to a review of the relevant context and terminology, before stating and proving the theorem.

Throughout the section, let $\complexes^n$ be equipped with the standard inner product.

\begin{definition}
A \defining{proper direct-sum decomposition} of $\complexes^n$ is an unordered set $\lambda$ of proper, non-zero, pairwise orthogonal subspaces of $\complexes^n$, whose sum is~$\complexes^n$. The subspaces are called the \defining{components} of $\lambda$. Given two direct-sum decompositions $\lambda$ and $\theta$, we say that $\lambda$ is a \defining{refinement} of $\theta$, and write $\lambda\le \theta$, if every component of~$\lambda$ is a subspace of some component of $\theta$.
\end{definition}

The set of direct-sum decompositions of $\complexes^{n}$ has a natural action of~$\Un$
and is therefore equipped with a natural topology as the disjoint union of its $\Un$-orbits. The partial ordering by refinement is respected by the $\Un$~action, so
the morphisms (unique between any two comparable objects) likewise have an inherited topology
via the $\Un$-action. The set of proper direct-sum decompositions therefore forms a \emph{topological poset}, internal to the category of topological spaces.

\begin{definition}
The space $\Lcal_n$ is the topological realization of the category of proper direct-sum decompositions of $\complexes^n$.
\end{definition}

The space $\Lcal_{n}$ was introduced in~\cite{Arone-Topology} and studied in detail in~\cite{Banff1, Banff2}.
The unreduced suspension of~$\Lcal_n$, denoted by~$\Lcal_n^\diamond$, is a building block in the study of the stable rank filtration of complex \Ktheory~\cite{Arone-Lesh-Crelle, Arone-Lesh-Fundamenta}.
Likewise, the Spanier-Whitehead dual of $\Lcal_n^\diamond$ plays an important role in the orthogonal calculus Taylor tower of the functor $V\mapsto BU(V)$ \cite{Arone-Topology}.
Computing the Bredon (co)homology of $\Lcal_{n}$ is our principal motivation for the current paper, which is part of a long-term program to understand the stable rank filtration, the Taylor tower for $V\mapsto BU(V)$, and the relationship between them.

The behavior of $\Lcal_{n}$ is fairly simple when $n$ is not a prime power, but the case $n=p^{k}$ is much more complicated, so we spend some time discussing the groups that turn out to be important in understanding~$\Lcal_{p^{k}}$.
First, the subgroup $\Gamma_k\subseteq \Upk$ is an extension of the center $S^{1}\subseteq\Upk$ by an elementary abelian \pdash group,
\begin{equation}\label{eq: SES}
1 \to S^1\to \Gamma_k \to \left(\field_{p}\right)^{2k}\to 1,
\end{equation}
and is therefore called a \defining{projective elementary abelian \pdash group}.
It appears in~\cite{Oliver-p-stubborn}, which gives an explicit matrix representation of~$\Gamma_{k}$, with generators given by permutation matrices and diagonal matrices. Although we do not need the representation theory of $\Gamma_{k}$ for our current work, it turns out that $\Gamma_{k}$ is (up to conjugacy) the unique projective elementary abelian \pdash subgroup of $\Upk$ whose action on $\complexes^{p^k}$ is irreducible. \cite[Section~6]{Banff2} discusses $\Gamma_{k}$ in detail,
largely from first principles.

Some elementary computations (\cite[Lemma~7.4]{Banff2}) show
that a short exact sequence such as \eqref{eq: SES} specifies a well-defined bilinear form: lift two elements of $(\field_{p})^{2k}$ to $\Gamma_{k}$ and compute their commutator, which is necessarily an element of $S^{1}$ of order~$p$.
The form is totally isotropic, i.e. every element is orthogonal to itself. Further, there is a one-to-one correspondence between isomorphism classes of groups that are central extensions of $S^{1}$ by $(\field_{p})^{2k}$ and isomorphism classes of totally isotropic bilinear forms on~$(\field_{p})^{2k}$ (\cite[Proposition~7.7]{Banff2}).
Recall that a \defining{symplectic} form is a non-degenerate, totally isotropic bilinear form. As a group, $\Gamma_{k}$ is characterized up to isomorphism among short exact sequences of the form \eqref{eq: SES} by giving rise to a symplectic form on $(\field_{p})^{2k}\cong\Gamma_{k}/S^1$.

Let $N\left(\Gamma_{k}\right)$ denote the normalizer of $\Gamma_{k}$ in~$\Upk$. The commutator form on $\Gamma_{k}/S^{1}$ takes values in the center $S^{1}\subseteq\Upk$, and as a result
the action of $N\left(\Gamma_{k}\right)$ on $\Gamma_{k}$ stabilizes the associated form. Hence there is a natural map from
$N\left(\Gamma_{k}\right)$ to~$\Symp{k}$, the \defining{symplectic group}, which is the group of form-preserving automorphisms of $\left(\field_p\right)^{2k}$ equipped with the standard symplectic form. It is easy to check that inner automorphisms of $\Gamma_{k}$ induce the identity map on $\left(\field_p\right)^{2k}$.
Further, \cite{Oliver-p-stubborn} establishes that all form-preserving automorphisms of $\left(\field_p\right)^{2k}$ can be realized by conjugations in~$\Upk$, so $N\left(\Gamma_{k}\right)$
fits into a short exact sequence
\begin{equation}     \label{eq: normalizer}
1\to \Gamma_k \to N(\Gamma_k) \to \Symp{k}\to 1.
\end{equation}
In particular, the Weyl group of $\Gamma_{k}$ in~$\Upk$ is the symplectic group~$\Symp{k}$ \cite[Theorem~6]{Oliver-p-stubborn}.

Lastly, recall that a subspace $V$ of a symplectic vector space is called \defining{coisotropic} if $V^{\bot}\subseteq V$. We say that $H\subseteq\Gamma_{k}$ is a \defining{coisotropic subgroup} if it is the inverse image of a coisotropic subspace of $\Gamma_{k}/S^1\cong (\field_{p})^{2k}$. Since the symplectic form on $(\field_{p})^{2k}$ is defined by commutators, an equivalent formulation is that 
$H\subseteq\Gamma_{k}$ is coisotropic if
$H$ contains its $\Gamma_{k}$-centralizer~$C_{\Gamma_{k}}(H)$. 
As a consequence, the conjugation action of $N\left(\Gamma_{k}\right)$ on $\Gamma_{k}$ permutes the coisotropic subgroups. However, the following lemma tells us that the action of $\Gamma_{k}\subseteq N\left(\Gamma_{k}\right)$ is actually trivial. 

\begin{lemma}   \label{lemma: coisotropic normal}
If $H\subseteq\Gamma_{k}$ is coisotropic, then $H$ is normal in~$\Gamma_{k}$. 
\end{lemma}

\begin{proof}
The short exact sequence \eqref{eq: SES} tells us that $\Gamma_{k}/S^{1}$ is abelian, and hence the commutator subgroup of $\Gamma_{k}$ is contained in~$S^{1}$. Further, by definition of coisotropic, $S^{1}\subseteq H$, and so in fact $[\Gamma_{k},\Gamma_{k}]\subseteq H$. An easy computation shows that $[G,G]\subseteq H\subseteq G$ implies $H\triangleleft G$.
\end{proof}

The geometric realization of the poset of proper coisotropic subspaces of $\left(\field_p\right)^{2k}$ is the Tits building for $\Symp{k}$, denoted $\TitsSymp{k}$. The space $\TitsSymp{k}$ is a finite complex with an action of~$\Symp{k}$. We can identify the poset of proper coisotropic subspaces of 
$\left(\field_p\right)^{2k}$ with the poset of proper coisotropic subgroups of~$\Gamma_{k}$, with the action of $N\left(\Gamma_{k}\right)/\Gamma_{k}\cong\Symp{k}$,
by Lemma~\ref{lemma: coisotropic normal}. 

Non-equivariantly, $\TitsSymp{k}$ is homotopy equivalent to a wedge of spheres. This follows from~\cite[Theorem 4.127]{Abramenko-Brown-Buildings}, which says that the homotopy type of every spherical building is a wedge of spheres, together with remarks at the beginning of \emph{loc.cit.} Section~4.7, according to which every finite building is spherical.
For details and further references regarding $\Gamma_k$ and its connections with the symplectic Tits building, see~\cite{Banff2} and~\cite{Arone-Lesh-Tits}.

\medskip
The calculation of the ordinary mod~$p$ homology of $\Lcal_{n}$ was undertaken in~\cite{Arone-Topology}.

\begin{theorem}\label{theorem: mod p homology}
\hfill
\begin{enumerate}
\item \cite[Theorem 4]{Arone-Topology} If $n$ is not a power of~$p$, then $\Lcal_n$ is mod $p$ acyclic.
\item \cite[Theorem 1]{Arone-Topology} If $n=p^k$, then there is a mod $p$ homology equivalence
\[
\Upk_{+}\wedge_{N\left(\Gamma_{k}\right)}\TitsSymp{k}^{\diamond}
\longrightarrow
\Lcal_{p^{k}}^{\diamond}.
\]
\end{enumerate}
\end{theorem}

The principal result of this paper, Theorem~\ref{theorem: application}, strengthens Theorem~\ref{theorem: mod p homology} to a result about Bredon (co)homology for appropriate coefficients.

\begin{theorem}\label{theorem: application}
\theoremapplicationtext
\end{theorem}

\begin{proof}
Let $P$ be a \pdash toral subgroup of~$\Un$. For simplicity in notation, we assume that $S^{1}\subseteq P$. This does not change the fixed point space of~$P$, because $S^{1}$ acts on any subspace of $\complexes^{n}$ by multiplication by scalars, hence acts trivially on~$\Lcal_n$.

First, consider the case when $n$ is not a power of $p$. We would like to apply Proposition~\ref{proposition: isomorphism},
for which we need to know that the fixed point space $\left(\Lcal_n\right)^P$ is mod~$p$ acyclic. In fact, \cite[Theorem 1.2]{Banff2} tells us that if
$P$ has the property that $\left(\Lcal_n\right)^P$ is not mod $p$ acyclic,
then $P$ fits into a short exact sequence
\[
1 \longrightarrow S^1 \longrightarrow P \longrightarrow P/S^1 \longrightarrow 1,
\]
where $S^1$ is the center of $\Un$, and $P/S^1$ is
an elementary abelian \pdash group. It follows that $\left(\Lcal_n\right)^P=(\Lcal_n^{S^1})^{P/S^1}$.
Thus $\left(\Lcal_n\right)^P=\left(\Lcal_n\right)^{P/S^1}$.

On the other hand, it follows from Theorem~\ref{theorem: mod p homology} that if $n$ is not a power of $p$, then $\Lcal_n$ is mod~$p$ acyclic.
Because $P/S^1$ is an elementary abelian \pdash group and
$\Lcal_{n}$ is a finite complex,
by Smith theory $\left(\Lcal_n\right)^{P/S^1}$ is mod~$p$ acyclic
\cite[Ch. III (4.21)--(4.30)]{MR889050}.
By Proposition~\ref{proposition: isomorphism}, it then follows that the map $\Lcal_{n} \to *$ induces an isomorphism on Bredon (co)homology with coefficients in~$M$. This proves the first part of the theorem.

It remains to tackle the case of~$\Lcal_{p^{k}}$.
Once again we want to apply Proposition~\ref{proposition: isomorphism}, so
we want to prove that the map
\begin{equation}  \label{eq: induction approximation map}
\Upk_{+}\wedge_{N\left(\Gamma_{k}\right)}\TitsSymp{k}^{\diamond}
\longrightarrow
\Lcal_{p^{k}}^{\diamond}
\end{equation}
induces a mod~$p$ homology isomorphism on the fixed point space of a \pdash toral subgroup~$P\subset\Upk$.

Suppose first that $P$ is not conjugate to a subgroup  of~$N\left(\Gamma_{k}\right)$. In this case the fixed points on the left side of \eqref{eq: induction approximation map}
are just a point. Because $P$ is not conjugate to a subgroup of~$\Gamma_{k}$, \cite[Theorem 1.2]{Banff2} tells us that $\left(\Lcal_{p^{k}}\right)^{P}\simeq *$, so the map of $P$-fixed points of \eqref{eq: induction approximation map}
is indeed a homotopy equivalence.

Next, suppose that $P$ is conjugate to a subgroup
of~$N\left(\Gamma_{k}\right)$.
For $P=\Gamma_{k}$ itself, the map of $P$-fixed points is actually a homeomorphism, by the calculation of \cite[Theorem 1.2]{Arone-Lesh-Tits}.
The same is true for \pdash toral subgroups of $N\left(\Gamma_{k}\right)$ that strictly contain~$\Gamma_{k}$.

Lastly, we consider \pdash toral subgroups $P\subset N\left(\Gamma_{k}\right)$ that do not contain~$\Gamma_{k}$, where we get a homology calculation by
using relative Smith theory. 
By \eqref{eq: normalizer},
$P$ is an extension of a finite \pdash subgroup by~$S^{1}$.
Let $C$ be the homotopy cofiber
of~\eqref{eq: induction approximation map},
which is mod~$p$ acyclic by
Theorem~\ref{theorem: mod p homology}.
The action of $S^1$ on both sides of
\eqref{eq: induction approximation map}
is trivial,
so as in the previous case, $C^{P}=C^{P/S^1}$.
But $P/S^1$ is a finite \pdash group, and $C$ is a finite, mod~$p$ acyclic complex.
As above, by Smith theory $C^{P}$ is mod $p$ acyclic.
Since fixed point spaces commute with homotopy cofibers~\cite[Proposition B.1(i)]{Schwede-Global}, it follows that the map~\eqref{eq: induction approximation map} induces a mod $p$ homology isomorphism on
the fixed point space of~$P$.

As in the previous case, by Proposition~\ref{proposition: isomorphism} we conclude that \eqref{eq: induction approximation map} induces an isomorphism on Bredon (co)homology with coefficients in $M$.
\end{proof}

\section{Satisfying the \upanddown/  condition for~$p$}
\label{section: up and down}

In this section we address the question of what (co)Mackey functors satisfy the \upanddown/ condition for~$p$. Our main result in this direction,
Proposition~\ref{proposition: up and down},
says that if a generalized (co)homology theory satisfies the \upanddown/ condition, then so does ordinary (co)homology with coefficients in the (co)Mackey functor associated with it. In Section~\ref{section: borel}, we will use Proposition~\ref{proposition: up and down} to show that the Borel (co)homology associated with a \pdash local $G$-spectrum $E$ satisfies the \upanddown/ condition.
Finally, in Section~\ref{section: LnBorel},
we will use this material to obtain specific information about the Bredon (co)homology groups of the space $\Lcal_n^\diamond$ with coefficients in a (co)Mackey functor corresponding to Borel homology.

Suppose that $E_*^G$ is a generalized equivariant homology theory that satisfies a slightly general version of the \upanddown/ condition of Definition~\ref{definition: upanddown}:
assume that for some fixed subgroup $P\subseteq G$ and all subgroups $H\subseteq G$, the composite
\[
E_*^G(G/H)
    \xrightarrow{\ \tr\ } E_*^G(G/P \times G/H)
    \longrightarrow E_*^G(G/H)
\]
is an isomorphism.
(In practice, the case of interest is when $P$ is
a maximal \pdash toral subgroup of~$G$.)
In this section we show that there is likewise a composite isomorphism for homology with coefficients
in a certain coMackey functor~$\pibar_*(E)$ associated with~$E_*^G$,
and likewise for cohomology with coefficients in a Mackey functor~$\piund^{*}(E)$.  More precisely, we will show in Proposition~\ref{proposition: up and down} that the following composite homomorphism is an isomorphism, for all $i$ and~$j$:
\[
\HH_i^G(G/H; \pibar_j(E))
      \xrightarrow{\ \tr\, }
\HH_i^G(G/P\times G/H; \pibar_j(E))
\longrightarrow
\HH_i^G(G/H; \pibar_j(E)).
\]
For a future application, we are especially interested in applying this to Borel (co)homology associated with a $G$-spectrum that is non-equivariantly \pdash local (see Proposition~\ref{proposition: p-local}).

Recall that generalized equivariant homology and cohomology theories are represented by (genuine)
$G$-spectra. Indeed, let $E$ be a $G$-spectrum. Then $E$ represents a homology theory on $G$-spectra by the formula
$E_*^G(X)\definedas\pi_*\left((E\wedge X)^G\right)$.
Dually, $E$ represents a cohomology theory via the formula
$E^*_G(X):=\pi_*\left(\map(X,E)^G\right)$.
In fact, every (co)homology theory on $G$-spectra is represented by a $G$-spectrum in this way~\cite[Chapter XIII Corollaries 3.3 and 3.5]{May-Alaska}.

Given a $G$-spectrum $E$, one can associate with it a coMackey and a Mackey functor, which we denote $\pibar_* E$ and~$\piund^{*}E$. They are given by the formulas
\begin{align}
\label{eq: pibar defn}
\pibar_*E\left(\Sigma^\infty G/H_+\right)&\definedas E_*^G\left(\Sigma^{\infty}G/H_+\right)
    \\
\nonumber
\piund^{*}E\left(\Sigma^\infty G/H_+\right)&\definedas E^*_G\left(\Sigma^{\infty}G/H_+\right).
\end{align}
One may
either think of $*$ as a fixed integer, or consider all values of $*$
simultaneously and view $\pibar_*$ and $\piund^{*}$ as (co)Mackey functors with values in
graded abelian groups. We adopt the latter point of view.

The (bigraded) (co)homology theories
$\HH^G_*\left(X; \pibar_*(E)\right)$ and
$\HH_G^*\left(X; \piund^{*}(E)\right)$ are the ordinary (co)homology theories that agree with $E_*^G$ and $E^*_G$, respectively, on the category~$\SO{G}$.
Hence if $X$ is a spectrum of the form $\Sigma^\infty G/H_+$, there are natural isomorphisms
\begin{equation}        \label{eq: natural isos}
\begin{array}{c}
\HH^G_0\left(X; \pibar_*(E)\right)
      \cong E^G_*(X)\\
\LARGEstrut\HH_G^0\left(X; \piund^{*}(E)\right)
      \cong E^*_G(X),
\end{array}
\end{equation}
and if $i>0$, then
$\HH^G_i\left(X; \pibar_*(E)\right)
\cong
\HH_G^i\left(X; \piund^{*}(E)\right)= 0$. More generally, these isomorphisms hold if $X$ is a wedge sum of spectra of the form $\Sigma^\infty G/H_+$.

In fact, we can say more. It is possible to compare the functors $\HH^G_0\left(X; \pibar_*(E)\right)$
      and $E^G_*(X)$.
\begin{lemma}    \label{lemma: comparison}
For $(-1)$-connected $G$-spectra, there are natural transformations
\begin{align*}
\HH^G_0\left(X; \pibar_*(E)\right) \longrightarrow E^{G}_*(X)\\
E_{G}^*(X)\longrightarrow \HH_G^0\left(X; \piund^{*}(E)\right),
\end{align*}
and they are isomorphisms if $X$ is a zero-dimensional CW-spectrum, i.e. a wedge sum of spectra of the form~$\Sigma^\infty G/H_+$.
\end{lemma}

\begin{proof}
Since $X$ is $(-1)$-connected, one can choose a CW-approximation for $X$ with no cells in negative dimensions. Let $X^0$ and $X^1$ be the $0$- and $1$-skeleta of $X$ respectively.
Composing the cofiber sequence boundary map
$\Sigma^{-1} X^1/X^0 \xrightarrow{\ \partial_1\ } X^0$ with inclusion to~$X$, we obtain a null-homotopic composite map
\[
\Sigma^{-1} X^1/X^0 \xrightarrow{\ \partial_1\ } X^0 \hookrightarrow X.
\]
Applying~$E^{G}_*$, we get a zero composite map of abelian groups,
and hence a natural factoring through the cokernel of $E^{G}_*(\partial_1)$:
\begin{equation}         \label{eq: cokernel diagram}
\begin{gathered}
\xymatrix@C+1pc{
E^{G}_*(\Sigma^{-1} X^1/X^0)
   \ar[r]^{\qquad E^{G}_*(\partial_1)}
   &E^{G}_*(X^0)
   \ar[r]\ar[d]
   & E^{G}_*(X)\\
& \cokernel E^{G}_*(\partial_1).
   \ar@{-->}[ur].
}
\end{gathered}
\end{equation}

Note that $X^0$ is equivalent to a finite direct sum of elements of~$\SO{G}$. It follows that the functor
$\Sigma^\infty O_+\mapsto \left[\Sigma^\infty O_+, X^0\right]_G$
is a finite sum of representable functors from $\SO{G}^{\op}$ to abelian groups.
By the coYoneda lemma, there is an isomorphism
\begin{equation}     \label{eq: E coend}
E^{G}_{*}(X^0)
\cong
E^{G}_{*}(\Sigma^\infty O_+)
     \otimes_{\SO{G}}
     \left[\Sigma^\infty O_+, X^0\right]_G.
\end{equation}
By definition \eqref{eq: pibar defn},  we have
$\pibar_{*}E(\Sigma^\infty O_+)=E^{G}_{*}(\Sigma^\infty O_+)$,
and \eqref{eq: E coend} becomes
\[
E^{G}_{*}(X^0)
\cong
\pibar_{*}E(\Sigma^\infty O_+)
   \otimes_{\SO{G}} \left[\Sigma^\infty O_+, X^0\right]_G.
\]
Similarly there is an isomorphism
\[
E^{G}_{*}\left(\Sigma^{-1}X^1/X^0\right)
\cong
\pibar_{*}E(\Sigma^\infty O_+)
     \otimes_{\SO{G}} \left[\Sigma^\infty O_+, \Sigma ^{-1} X^1/X^0\right]_G.
\]
It follows that the cokernel of $E^{G}_*(\partial_1)$ is the zero-th homology of the chain complex defined in~\eqref{eq: construction}. Therefore this cokernel is, by definition, canonically isomorphic to $\HH^G_0\left(X; \pibar_*(E)\right)$, and the dashed arrow in \eqref{eq: cokernel diagram} is the map required by the lemma.

If $X$ is zero-dimensional, one can choose $X^0=X$. The dashed arrow in
\eqref{eq: cokernel diagram} becomes an isomorphism, and we have
$\HH^G_0\left(X; \pibar_*(E)\right) \cong E^{G}_*(X)$.

The proof of the dual cohomological statement is the same.
\end{proof}

\begin{remark}     \label{remark: natural restriction}
It follows from Lemma~\ref{lemma: comparison} that if $X$ is a $G$-space (not spectrum), there is a homomorphism $\HH^G_0\left(X; \pibar_*(E)\right) \to E^{G}_*(X)$,
which is natural with respect to
stable maps in the variable $X$. We simply restrict the domain category from $(-1)$-connected spectra to suspension spectra.
\end{remark}

\subsubsection*{The Becker-Gottlieb transfer}
Suppose $\eta\colon E\to B$ is a $G$-equivariant fibration whose fiber is a finite complex~$F$.
The equivariant Becker-Gottlieb transfer
of $\eta$ is a map of $G$-spectra $\tr\colon \Sigma^\infty B_+ \to \Sigma^\infty E_+$~\cite{waner}. Therefore, if~$h_*$ is a homology theory defined on $G$-spectra (or in other words, an $RO(G)$-graded homology theory), the fibration $\eta$ gives rise to a transfer homomorphism $h_*(B)\to h_*(E)$. In particular the transfer homomorphism is defined if $h_*
(\whatever)=\HH_*^G(\whatever; M)$ is Bredon homology with coefficients in a coMackey functor.

The following is a key property of the transfer map.

\begin{lemma}[Becker-Gottlieb~\cite{becker-gottlieb}]\label{lemma: euler}
The composition
$\Sigma^\infty B_+
     \xrightarrow{\ \tr\ } \Sigma^\infty E_+
     \xrightarrow{\ \eta\ } \Sigma^\infty B_+$
induces multiplication by the Euler characteristic of the fiber~$F$ on non-equivariant homology.
\end{lemma}

Now we can state and prove the main result of this section.

\begin{proposition}\label{proposition: up and down}
Let $G$ be a compact Lie group and let $P\subseteq G$ be a subgroup. Let $E$ be a $G$-spectrum. Suppose that for every $G$-orbit $G/H$, the composite map
\[
\Sigma^\infty G/H_+ \xrightarrow{\,\tr\,} \Sigma^\infty (G/P\times G/H)_+
                  \longrightarrow \Sigma^\infty G/H_+
\]
induces an isomorphism
\[
E_*^G(G/H) \xrightarrow{\ \tr\ } E_*^G(G/P \times G/H)\longrightarrow E_*^G(G/H).
\]
Then the induced homomorphism on homology with coefficients in $\pibar_*(E)$ is an isomorphism as well. That is, for all $i, j$ the following composite homomorphism is an isomorphism:
\[
\HH_i^G(G/H; \pibar_j(E))
      \xrightarrow{\ \tr\, }
\HH_i^G(G/P\times G/H; \pibar_j(E))
\longrightarrow
\HH_i^G(G/H; \pibar_j(E)).
\]
The analogous result for cohomology with coefficients in $\piund^{*}(E)$ also holds.
\end{proposition}

\begin{proof}
For $i>0$ there is nothing to prove, because
$\HH_i^G(G/H; \pibar_j(E))=0$, so we need only handle $i=0$.

We consider the following commutative diagram, where the vertical homomorphisms are
given by the natural transformation of Lemma~\ref{lemma: comparison} (see also Remark~\ref{remark: natural restriction}):
\begin{equation}    \label{diagram: rectangular}
\small{
\begin{array}{ccccc}
\HH_0^G(G/H; \pibar_*(E)) & \longrightarrow & \HH_0^G(G/P\times G/H; \pibar_*(E)) & \longrightarrow & \HH_0^G\left(G/H; \pibar_*(E)\right)\\
\begin{sideways}$\xleftarrow{\cong}$ \end{sideways}
      & &\begin{sideways}$\leftarrow$\end{sideways} & &
      \begin{sideways}$\xleftarrow{\cong}$ \end{sideways} \\
E_*^G(G/H) & \longrightarrow & E_*^G(G/P\times G/H) & \longrightarrow & E_*^G(G/H).
\end{array}
}
\end{equation}
By assumption, the composition in the lower row is an isomorphism. It follows that the composition in the upper row is an isomorphism.

This proves the proposition for $i=0$, and the proof is complete for the case of homology. The proof of the cohomology case is the same.
\end{proof}

\begin{remark*}
Note that if $G$ is finite group, then the vertical map in the middle of the diagram is also an isomorphism. This is because in this case a product
$G/P\times G/H$ is a disjoint union of $G$-orbits.
But when $G$ is a non-finite compact Lie group, then in most cases $G/P \times G/H$ is not a disjoint union of $G$-orbits,
and the middle map is not an isomorphism.
\end{remark*}


\section{Borel homology}\label{section: borel}
In this section we describe an example, a family of generalized equivariant homology theories and their associated (co)Mackey functors. This example is needed for the intended later applications.

If $Y$ is a pointed $G$-space, we
define the \defining{based homotopy orbit space} of $Y$ by
$Y_{\hobased G}\definedas (EG\times_{G}Y)/(EG\times_{G}*)$.
Given a $G$-spectrum~$E$, we then define the \defining{Borel homology associated with~$E$} by the formula
\[
E_*^{bG}(X)\definedas \pi_*(E\wedge X)_{\hobased G}.
\]
Borel homology is an $RO(G)$-graded homology theory.
Indeed, as a result of the Adams isomorphism,
\begin{equation}\label{eq: adams}
\left(\Sigma^{-\ad_G}E\wedge EG_{+}\wedge X\right)^G\simeq (E\wedge X)_{\hobased G}.
\end{equation}
Hence Borel homology is represented in the usual way by the free $G$-spectrum
$\Sigma^{-\ad_G}E \wedge EG_+$, where $\Sigma^{-\ad_G}E$ denotes the desuspension of $E$ by the adjoint representation. (See~\cite[Ch. XVI Theorem 5.4]{May-Alaska}.)

It is worth noting that Borel homology is well defined for naive $G$-spectra.
That is, the Borel homology functor factors through the forgetful functor from
the homotopy category of genuine $G$-spectra
to the homotopy category of naive $G$-spectra.
Thus,
if $f\colon X\to Y$ is a $G$-map that is a non-equivariant equivalence, then $f$ induces an isomorphism in Borel homology, because
the  homotopy category of naive $G$-spectra is the localization of the genuine homotopy category with respect to equivalences of this type.
Another pleasant invariance property of Borel homology
is that if $f$ induces an isomorphism $E_*(X)\to E_*(Y)$ of non-equivariant $E$-homology, then $f$ induces an isomorphism on Borel homology.

Dually, the \defining{Borel cohomology associated with~$E$} is defined by the formula
\[E^*_{bG}(X)
   \definedas \pi_*\map(X, E)_{\hobased G}.
\]
Note that when $X$ is a homotopy finite spectrum (which is true in all examples we ever consider),  $E^*_{bG}(X)\cong E_*^{bG}(DX)$, where $DX$ denotes the Spanier-Whitehead dual of~$X$. Borel cohomology satisfies analogues of the properties of Borel homology.

\begin{definition}\label{definition: borel mackey}
Given a $G$-spectrum~$E$,
the (co)Mackey functor associated with Borel (co)homology will be denoted by $\pibar^{b}_*(E)$ and $\piund^{*}_{b}(E)$. They are defined by the formulas
\begin{align*}
\pibar^{b}_*(E)(\Sigma^\infty G/H_+) &\definedas\pi_*\left(E\wedge G/H_+\right)_{\hobased G}\\
                                     &\cong \pi_* E_{\hobased H}\\
\intertext{and}
\piund^{*}_{b}(E)(\Sigma^\infty G/H_+)&\definedas\pi_*\left(\largestrut E\wedge D(G/H_+)\right)_{\hobased G}\\
                         &\cong \pi_*\left(\Sigma^{\ad_H-\ad_G} E\right)_{\hobased H}.
\end{align*}
\end{definition}

\begin{remark}\label{remark: restriction}
Recall that if $M$ is a (co)Mackey functor for $G$
and $H\subseteq G$, then $M\restrictedto{H}$ is the (co)Mackey functor defined by $M\restrictedto{H}(\Sigma^\infty H/K_+)=M(\Sigma^\infty G/K_+).$ Similarly, if $E$ is a $G$-spectrum,
we write $E\restrictedto{H}$ for the spectrum $E$ with action restricted from $G$ to~$H$.
Later on, in the proof of Proposition~\ref{proposition: L_n}, we will need
to compare the restricted Borel (co)Mackey functors
$(\pibar^{b}_* E)\restrictedto{H}$ and $(\piund_{b}^* E)\restrictedto{H}$
with the Borel (co)Mackey functors obtained by restricting $E$ to~$H$, namely
$\pibar^b_*(E\restrictedto{H})$ and $\piund_b^*(E\restrictedto{H})$.
In the coMackey case, the two restrictions coincide, but in the Mackey case there is a dimension shift by $\ad_G-\ad_H$, as follows.
\begin{enumerate}
\item
The restriction of the coMackey functor $\pibar^{b}_* E$ to $H$ is given by the formula
\[
(\pibar^{b}_* E)\restrictedto{H}
\cong
\pibar^b_*(E\restrictedto{H}).
\]
To see this, suppose that $\Sigma^\infty H/K_+$ is an object of~$\SO{H}$.
We have isomorphisms
\begin{align*}
(\pibar^{b}_* E)\restrictedto{H}(\Sigma^\infty H/K_+)
   &\cong \pibar^{b}_*(E)(\Sigma^\infty G/K_+)\\
\pibar^b_*(E\restrictedto{H})(\Sigma^\infty H/K_+)
    &\cong \pi_*\left(E\restrictedto{H}\wedge H/K_+\right)_{\hobased H}
\end{align*}
and the right-hand sides are both isomorphic to $\pi_*(E_{\hobased K})$.
\item
By contrast, the restriction of the Mackey functor $\piund^{*}_{b}E$ to $H$ is given by the formula
\[
(\piund^{*}_{b}E)\restrictedto{\strut H}
\cong
\piund^b_*\left(\Sigma^{\ad_H-\ad_G}E\restrictedto{H}\right).
\]
We can see this by comparing the following two chains of isomorphisms
and observing that the right-hand sides differ by a suspension of $\ad_H-\ad_G$:
\begin{align*}
(\piund^{*}_{b}E)\restrictedto{\strut H}(\Sigma^\infty H/K_+)
   &\cong (\piund^{*}_{b}E)(\Sigma^\infty G/K_+)
   \cong \pi_*\left(\Sigma^{\ad_K-\ad_G} E\right)_{\hobased K}
\\
\piund^b_*\left(E\restrictedto{H}\right)(\Sigma^\infty H/K_+)
   &\cong
\pi_*\left(\Sigma^{\ad_K-\ad_H} E\right)_{\hobased K}.
\end{align*}
\end{enumerate}
\end{remark}

\bigskip

The proposition below will be needed in future applications. It verifies that Borel homology associated with a \pdash local spectrum satsifies the \upanddown/ condition for~$p$.

\begin{remark}\label{remark: p-locality}
We remind the reader that a spectrum is \defining{\pdash local} if and only if its homotopy groups are \pdash local abelian groups. Furthermore, \pdash localization is a smashing localization, and so a homotopy colimit of \pdash local spectra is \pdash local~\cite[Proposition 2.4]{bousfield-localization-spectra}.
\end{remark}

\begin{proposition}\label{proposition: p-local}
Fix a prime $p$. Let $E$ be a \pdash local spectrum with an action of~$G$.
Let $P$ be a maximal \pdash toral subgroup
of~$G$. Then the coMackey functor $\pibar^{b}_*E$ associated with Borel homology of $E$ satisfies the \upanddown/ condition for the prime~$p$.
More explicitly, for every subgroup $H\subseteq G$, the following composite homomorphism is an isomorphism for all $i$ and~$j$:
\[
\HH_i^G(G/H; \pibar^b_j E)
      \xrightarrow{\ \tr\, }
\HH_i^G(G/P\times G/H; \pibar^b_j E)
      \longrightarrow
\HH_i^G(G/H; \pibar^b_j E).
\]
The analogous result for cohomology with coefficients in the Mackey functor $\piund^b_* E$ also holds.
\end{proposition}

\begin{proof}
We saw that by virtue of the Adams isomorphism~\eqref{eq: adams}, the spectrum $\Sigma^{-\ad_G}E \wedge EG_+$ represents the Borel homology theory~$E^{bG}_*$.
We claim that the spectrum $\Sigma^{-\ad_G}E \wedge EG_+$ satisfies the hypothesis of Proposition~\ref{proposition: up and down}.

To justify this claim, we will check that for {\it every} $G$-space~$X$,
the following composite map is an equivalence:
\[
\left(\largestrut E\wedge X_+\right)_{\hobased G}
\xrightarrow{\ \tr\ }
\left(\largestrut E\wedge
   \left(\strut G/P\times X\right)_+
   \right)_{\hobased G}
\longrightarrow
\left(\largestrut E\wedge X_+\right)_{\hobased G}.
\]
It suffices to show that we have a non-equivariant equivalence (before taking homotopy orbits)
\begin{equation}   \label{eq: first sequence}
 E\wedge X_+
\xrightarrow{\ \tr\ }
 E\wedge
   \left(\mbox{\large\strut} G/P\times X\right)_+
\longrightarrow
 E\wedge X_{+}.
\end{equation}
This sequence is the smash product of $E\wedge X_+$ with the composite map
\begin{equation}   \label{eq: sequence}
 \Sigma^\infty S^0
\xrightarrow{\ \tr\ }
\Sigma^\infty G/P_+
\longrightarrow
\Sigma^\infty S^0.
\end{equation}

By Lemma~\ref{lemma: euler}, the composite map \eqref{eq: sequence} induces multiplication by the Euler characteristic $\chi(G/P)$ on homology. But $\chi(G/P)$ is invertible mod~$p$ because $P$~is a maximal \pdash toral subgroup of~$G$. It follows that~\eqref{eq: sequence} is a \pdash local equivalence. Since $E$ is assumed to be \pdash local, taking the smash product of
\eqref{eq: sequence} with $E\wedge X_+$ yields a composite equivalence in~\eqref{eq: first sequence}.

The proof of the dual result for cohomology is the same.
\end{proof}


\section{Bredon (co)homology of $\Lcal_{n}^\diamond$ with Borel coefficients}\label{section: LnBorel}

This section is devoted to a calculation of Bredon (co)homology of $\Lcal_n^\diamond$ needed in a forthcoming paper where the results of this paper will be applied. More specifically,
we are interested the (co)homology of $\Lcal_n^\diamond$ with coefficients in a (co)Mackey functor associated with Borel homology.

Recall from Section~\ref{section: borel} that if $E$ is a spectrum with an action of~$G$, then the Borel homology associated
with~$E$ is given by
\[
E_*^{bG}(X)\definedas \pi_*(E\wedge X)_{\hobased G},
\]
and there is an associated coMackey functor $(\pibar^{b}_{*}E)(\whatever)$,
defined by
\begin{equation}   \label{eq: cMF again}
(\pibar^{b}_*E)(\Sigma^\infty G/H_+)
   \definedas\pi_*\left(E\wedge G/H_+\right)_{\hobased G}
   \cong \pi_* E_{\hobased H}.
\end{equation}
We seek information about
$\HH_i^{\Unup}\left(\Lcal_n^\diamond; \pibar^b_* E\right)$
when $E$ is a \pdash local spectrum (Proposition~\ref{proposition: L_n} below).
The computation when $n\neq p^k$ will be straightforward from Theorem~\ref{theorem: application}. However, the case $n=p^k$ is more difficult, and we need some preparations that will allow us to change the groups over which we compute Bredon (co)homology.

Consider the following general set up. Let $G$ be a compact Lie group,
with a normal subgroup~$\Gamma\triangleleft G$,
quotient $S= G/\Gamma$,
and quotient map $q\colon G \to S$.
Then $q$ induces a pullback functor
$q^{!}\colon \top_S\to \top_G$, which has a derived left adjoint $q_{h!}\colon \top_G\to \top_S$, given by taking homotopy orbits of the action of~$\Gamma$.
Similarly, there also is a pullback functor $q^{!}$ from the category of naive $S$-spectra to naive $G$-spectra. This functor has a derived left adjoint too, which we will describe explicitly now, since it plays a role in Lemma~\ref{lemma: pushforward} below.
By using the model $E_{\hobased \Gamma}=(E\wedge EG_+)_{\Gamma}$,
we see that $E_{\hobased \Gamma}$ has an action of the quotient
group~$S=G/\Gamma$, and $(E_{\hobased \Gamma})_{\hobased S}\simeq E_{\hobased G}$.
Hence the $S$-spectrum $E_{\hobased \Gamma}$ gives us a coMackey functor $\pibar^b_* E_{\hobased \Gamma}$
for the group~$S$, just as the $G$-spectrum~$E$
gives us the coMackey functor $\pibar^b_* E$ for
the group~$G$~\eqref{eq: cMF again}.
Dually, we have Mackey functors $\piund_b^* E$ and
$\piund_b^*E_{\hobased \Gamma}$ for $G$ and~$S$, respectively (Definition~\ref{definition: borel mackey}). This setup gives us the following adjunction result.

\begin{lemma}\label{lemma: pushforward}
Let $\Gamma, G$ and $S$ be as above, and let $E$ be a spectrum with an action of~$G$.
Let $X$ be a pointed $S$-space. There is an isomorphism of bigraded Bredon homology groups, natural in~$X$,
\[
\HHwiggle_{*}^{G}\left(q^{!}(X); \pibar^b_* E\right)\cong \HHwiggle_{*}^{S}\left(X; \pibar^b_* E_{\hobased \Gamma}\right).
\]
Dually there is an isomorphism of Bredon cohomology groups
\[
\HHwiggle^{*}_{G}\left(q^{!}(X); \piund_b^*E\right)\cong \HHwiggle^{*}_{S}\left(X; \pibar_b^*E_{\hobased \Gamma}\right).
\]
\end{lemma}

\begin{proof}
As usual, we will focus on the homology case. The proof of the cohomology case is very similar.

By Definition~\ref{definition: bredon}, the homology groups $\HHwiggle_{*}^{G}\left(q^{!}(X); \pibar^b_*E\right)$ are the homology groups of the following homotopy coend of chain-complex valued functors:
\[
\pibar^b_* E  \htensor_{\nc \Ocal_G} \nc \FixedPtFunctor{q^{!}(X)}.
\]
Similarly, the homology groups $\HHwiggle_{i}^{S}\left(X; \pibar^b_* E_{\hobased \Gamma}\right)$ are given by the
homology groups of the homotopy coend
\[
\pibar^b_* E_{\hobased \Gamma} \htensor_{\nc \Ocal_S} \nc \FixedPtFunctor{X}.
\]
The functor $q^{!}\colon \Ocal_S \to \Ocal_G$ induces a functor of differential graded categories, denoted $\nc q^{!}\colon  \nc\Ocal_S \to \nc\Ocal_G$.
Moreover, for any orbit $Q$ of $S$, i.e., for any object $Q$ of~$\Ocal_S$, there are natural isomorphisms
\begin{align}\nonumber
(\pibar^b_* E_{\hobased \Gamma})(Q)
&=\pi_*\left(
    E_{\hobased \Gamma}\wedge Q_+\right)_{\hobased S}
            \\ \label{eq: pullback}
&\xrightarrow{\cong}
\pi_*\left(E\wedge (q^{!}Q)_+\right)_{\hobased G}
          \\ \nonumber
&= (\pibar^b_* E)\left(q^{!}(Q)\right).
\end{align}
And similarly
\begin{align*}
\FixedPtFunctor{X}(Q)&= \map_S(Q, X)\\
                     &\xrightarrow{\cong} \map_G\left(q^{!}(Q), q^{!}(X)\right) \\
                     &= \FixedPtFunctor{q^{!}(X)}\left(q^{!}(Q)\right).
\end{align*}
These isomorphisms, together with the functor~$\nc q^{!}$, induce a homomorphism of homotopy coends
\begin{equation}  \label{eq: map of coends}
\pibar^b_* E_{\hobased \Gamma}  \htensor_{\nc \Ocal_S} \nc \FixedPtFunctor{X}
\longrightarrow
\pibar^b_* E  \htensor_{\nc \Ocal_G} \nc \FixedPtFunctor{q^{!}(X)}.
\end{equation}
We need to prove that \eqref{eq: map of coends} is a homotopy equivalence
(i.e., a quasi-isomorphism).

By~\eqref{eq: pullback}, the functor
$\pibar^b_* E_{\hobased \Gamma} $ is the pullback of the functor $\pibar^b_* E$ along the functor
$\nc q^{!}\colon  \nc\Ocal_S \to \nc\Ocal_G$.
To prove that the map \eqref{eq: map of coends} is an equivalence, it is enough to prove that the functor $\FixedPtFunctor{q^{!}(X)}$ is equivalent to the homotopy left Kan extension of the functor  $\FixedPtFunctor{X}$ along the functor~$q^{!}\colon \Ocal_S^{\op} \to \Ocal_G^{\op}$.
In fact one could easily prove this via the derived enriched coYoneda lemma (Lemma~\ref{lemma: coyoneda}), but we will use a more direct argument.

We recall that for a $G$-orbit~$O$ and $S$-orbit~$Q$, we have
\begin{align*}
\FixedPtFunctor{q^{!}(X)}(O)
      &=\map_G\left(\strut O, q^{!}(X)\right)\\
\FixedPtFunctor{X}(Q)=\map_S(Q, X)
      &=\map_G\left(\strut q^{!}(Q), q^{!}(X)\right).
\end{align*}
By using the pointwise definition of the homotopy Kan extension of~$\FixedPtFunctor{X}$,
we find that
we need to prove that for
every~$O\in \Ocal_G$, the following natural map is an equivalence:
\begin{equation}\label{eq: local assembly}
\underset{O\to q^{!}(Q)}{\hocolim} \map_G\left(q^{!}(Q), q^{!}(X)\right)
\longrightarrow
\map_G\left(O, q^{!}(X)\right).
\end{equation}

The indexing category for the homotopy colimit (of a contravariant functor) at the source of~\eqref{eq: local assembly} has objects that are arrows in $\Ocal_{G}$
from the $G$-orbit~$O$ to the pullback of an $S$-orbit~$Q$; the morphisms are the obvious commuting triangles.
But the category of such arrows has an initial object, namely the arrow $O\to O/\Gamma$. It follows that the homotopy colimit in~\eqref{eq: local assembly} is naturally equivalent to~$\map_G\left(O/\Gamma, q^{!}(X)\right)$.
To finish the proof of the lemma, we observe that
the map
\[
\map_G(O/\Gamma, q^{!}(X))\to \map_G(O, q^{!}(X))
\]
is a homeomorphism because the action of $\Gamma$ on $q^{!}(X)$ is trivial.
\end{proof}

Now we are ready to do the calculation of the Bredon (co)homology 
of~$\Lcal_n^\diamond$.

\begin{proposition}   \label{proposition: L_n}
\propositionLntext
\end{proposition}

\begin{proof}
Since $E$ is \pdash local,
by Proposition~\ref{proposition: p-local} we know the (co)Mackey functors $\pibar^b_*(E)$ and $\piund^b_*(E)$ satisfy the \upanddown/ condition for the prime~$p$. Further, these (co)Mackey functors take values in \pdash local groups (Remark~\ref{remark: p-locality}).
Therefore they satisfy the assumptions of Theorem~\ref{theorem: application}.
It follows that if $n$ is not a power of $p$,
then the map $\Lcal_n \to *$ induces an isomorphism on the
Bredon homology groups $\HH_i^{\Unup}\left(\whatever; \pibar^b_*(E)\right)$, and similarly for cohomology.
This is equivalent to saying that $\HHwiggle_i^{\Unup}\left(\Lcal_n^\diamond; \pibar^b_*(E)\right)= 0$, and similarly for cohomology, establishing the first statement in the proposition.

Now suppose that $n=p^k$. We will focus on the homology case.
The proof of the cohomology case is essentially the same,
with the tweak that one needs to keep track of the adjoint representation sphere when applying Spanier-Whitehead duality to a compact Lie group,
as per Remark~\ref{remark: restriction}.
This does not affect the logic of the proof.

We pause for a notational comment. As we proceed through the proof, we restrict the group action of $\Upk$ on~$E$ (or its homotopy orbits) to smaller and smaller subgroups. Carrying through the restriction notation becomes cumbersome and unilluminating, and so from now on we trust the reader to observe what restriction has been taken by
looking at the group indicated in the notation for Bredon homology.

We need to show that $\HHwiggle_i^{\Upkup}\left(\Lcal_{p^k}^\diamond;
        \pibar^b_* E\right)=0$ for $i\ne k$.
By second part of Theorem~\ref{theorem: application}, followed by a reduced
version of Lemma~\ref{lemma: induction}, there are isomorphisms
 \begin{align}
 \nonumber
 \HHwiggle_i^{\Upkup}
       \left(\Lcal_{p^k}^\diamond; \pibar^b_* E\right)
 &\cong \HHwiggle_i^{\Upkup}
       \left(\Upk_{+}\wedge_{N\left(\Gamma_{k}\right)}
          \TitsSymp{k}^{\diamond};\pibar^b_* E\right)
\\
\label{eq: Tits homology}
&
\cong  \HHwiggle_i^{N\left(\Gamma_{k}\right)}
       \left(\TitsSymp{k}^{\diamond};
       \pibar^b_*E
       \right).
\end{align}
We focus on the last group, and our goal is to show that it is zero except (possibly) when $i=k$.

Our next step is to leverage the triviality of the action of 
$\Gamma_{k}\subseteq N\left(\Gamma_{k}\right)$
on the $N\left(\Gamma_{k}\right)$-space~$\TitsSymp{k}^{\diamond}$, 
as discussed in Lemma~\ref{lemma: coisotropic normal} and the paragraphs following it.
We apply Lemma~\ref{lemma: pushforward} with the group
$G=N\left(\Gamma_{k}\right)$, the subgroup $\Gamma=\Gamma_k$,
and the quotient $S=N\left(\Gamma_{k}\right)/\Gamma_k=\Symp{k}$,
acting on $X=\Lcal_{p^k}^\diamond$.
By Lemma~\ref{lemma: pushforward} applied
at~\eqref{eq: Tits homology}, we find
\begin{equation}   \label{eq: what we want}
\HHwiggle_i^{N\left(\Gamma_{k}\right)}
       \left(\TitsSymp{k}^{\diamond}; \pibar^b_* E
       \right)
\cong
\HHwiggle_i^{\Symp{k}}
   \left(\TitsSymp{k}^{\diamond};
       \pibar^b_* E_{\hobased \Gamma_k}\right).
\end{equation}
Hence we must show that the right side of
\eqref{eq: what we want} is zero for $i\ne k$.

The group
$\Symp{k}$ is finite, and we write $\Syl_p$ for a choice of
a Sylow \pdash subgroup. Recall that \pdash local spectra are preserved by homotopy colimits (Remark~\ref{remark: p-locality}). In particular, if the spectrum $E$ is \pdash local, then so is $E_{\hobased \Gamma_k}$. This fact  allows us to apply Proposition~\ref{proposition: p-local} to deduce that Bredon  homology with coefficients in the Borel coMackey functor $\pibar^b_*(E_{\hobased \Gamma_k})$ has the property that transferring to the Sylow \pdash subgroup and back gives an isomorphism:
\begin{align*}
\HHwiggle_i^{\Symp{k}}
      \left(\Largestrut\whatever ; \pibar^b_* E_{\hobased \Gamma_k}\right)
&\xrightarrow{\ \tr\ }
\HHwiggle_i^{\Syl_p}
      \left(\whatever ; \pibar^b_* E_{\hobased \Gamma_k}
      \right)\\
&\longrightarrow
\HHwiggle_i^{\Symp{k}}
      \left(\Largestrut\whatever ;
      \pibar^b_* E_{\hobased \Gamma_k}\right).
\end{align*}
It follows that the groups
$\HHwiggle_i^{\Symp{k}}\left(\TitsSymp{k}^{\diamond};
     \pibar^b_* E_{\hobased \Gamma_k}\right)$
are direct summands of those taken
over the Sylow \pdash subgroup, $\HHwiggle_i^{\Syl_p}\left(\TitsSymp{k}^{\diamond};
     \pibar^b_*(E_{\hobased \Gamma_k})
     \right)$.

But from the
proof of~\cite[Theorem 4.127]{Abramenko-Brown-Buildings},
we know that, as a $\Syl_p$-space, the Tits building $\TitsSymp{k}$ is equivalent to $(\Syl_p)_+\wedge S^k$,
i.e., a bouquet of $k$-spheres freely permuted by~$\Syl_p$.
We conclude that there are isomorphisms
\begin{align}
\label{eq: the whole thing}
\HHwiggle_i^{\Syl_p}
     \left(\TitsSymp{k}^{\diamond};
    \pibar_* E_{\hobased \Gamma_k}
    \right)
&\cong \HHwiggle_i^{\Syl_p}
     \left((\Syl_p)_+\wedge S^k;
     \pibar_* E_{\hobased \Gamma_k}
     \right)\\
\label{eq: ordinary}
&\cong \HHwiggle_i
     \left(S^k;
     \pi_* E_{\hobased \Gamma_k}\right).
\end{align}
Here \eqref{eq: ordinary} means the ordinary, non-equivariant, reduced homology groups of the sphere $S^k$ with coefficients in $\pi_* E_{\hobased \Gamma_k}$, and these homology groups are clearly zero for $i\ne k$. Since, as indicated above, the left side of \eqref{eq: the whole thing} contains the group
\[
\HHwiggle_i^{\Symp{k}}
   \left(\TitsSymp{k}^{\diamond};
   \pibar^b_* E_{\hobased \Gamma_k}\right)
\]
as a direct summand, it follows that the latter
group is also zero for $i\ne k$.
\end{proof}

\bibliographystyle{amsalpha}

\begin{thebibliography}{BHN{\etalchar{+}}19}

\bibitem[AB08]{Abramenko-Brown-Buildings}
Peter Abramenko and Kenneth~S. Brown, \emph{Buildings}, Graduate Texts in
  Mathematics, vol. 248, Springer, New York, 2008, Theory and applications.
  \MR{2439729}

\bibitem[AD01]{Arone-Dwyer}
G.~Z. Arone and W.~G. Dwyer, \emph{Partition complexes, {T}its buildings and
  symmetric products}, Proc. London Math. Soc. (3) \textbf{82} (2001), no.~1,
  229--256. \MR{1794263 (2002d:55003)}

\bibitem[ADL08]{ADL1}
Gregory~Z. Arone, William~G. Dwyer, and Kathryn Lesh, \emph{Loop structures in
  {T}aylor towers}, Algebr. Geom. Topol. \textbf{8} (2008), no.~1, 173--210.
  \MR{2377281}

\bibitem[ADL16]{ADL2}
G.~Z. Arone, W.~G. Dwyer, and K.~Lesh, \emph{Bredon homology of partition
  complexes}, Doc. Math. \textbf{21} (2016), 1227--1268. \MR{3578208}

\bibitem[AL07]{Arone-Lesh-Crelle}
Gregory~Z. Arone and Kathryn Lesh, \emph{Filtered spectra arising from
  permutative categories}, J. Reine Angew. Math. \textbf{604} (2007), 73--136.
  \MR{2320314 (2008c:55013)}

\bibitem[AL10]{Arone-Lesh-Fundamenta}
\bysame, \emph{Augmented {$\Gamma$}-spaces, the stable rank filtration, and a
  {$bu$} analogue of the {W}hitehead conjecture}, Fund. Math. \textbf{207}
  (2010), no.~1, 29--70. \MR{2576278}

\bibitem[AL20]{Arone-Lesh-Tits}
Gregory Arone and Kathryn Lesh, \emph{Fixed points of coisotropic subgroups of
  {$\Gamma_k$} on decomposition spaces}, Homology Homotopy Appl. \textbf{22}
  (2020), no.~1, 77--96. \MR{4027291}

\bibitem[Aro02]{Arone-Topology}
Greg Arone, \emph{The {W}eiss derivatives of {$B{\rm O}(-)$} and {$B{\rm
  U}(-)$}}, Topology \textbf{41} (2002), no.~3, 451--481. \MR{1910037}

\bibitem[BG76]{becker-gottlieb}
J.~C. Becker and D.~H. Gottlieb, \emph{Transfer maps for fibrations and
  duality}, Compositio Math. \textbf{33} (1976), no.~2, 107--133. \MR{0436137}

\bibitem[BHN{\etalchar{+}}19]{Balmer-spectrum}
Tobias Barthel, Markus Hausmann, Niko Naumann, Thomas Nikolaus, Justin Noel,
  and Nathaniel Stapleton, \emph{The {B}almer spectrum of the equivariant
  homotopy category of a finite abelian group}, Invent. Math. \textbf{216}
  (2019), no.~1, 215--240. \MR{3935041}

\bibitem[BJL{\etalchar{+}}15]{Banff1}
Julia~E. Bergner, Ruth Joachimi, Kathryn Lesh, Vesna Stojanoska, and Kirsten
  Wickelgren, \emph{Fixed points of {$p$}-toral groups acting on partition
  complexes}, Women in topology: collaborations in homotopy theory, Contemp.
  Math., vol. 641, Amer. Math. Soc., Providence, RI, 2015, pp.~83--96.
  \MR{3380070}

\bibitem[BJL{\etalchar{+}}19]{Banff2}
\bysame, \emph{Classification of problematic subgroups of {$U(n)$}}, Trans.
  Amer. Math. Soc. \textbf{371} (2019), no.~10, 6739--6777. \MR{3939560}

\bibitem[BK72]{Bousfield-Kan}
A.~K. Bousfield and D.~M. Kan, \emph{Homotopy limits, completions and
  localizations}, Lecture Notes in Mathematics, Vol. 304, Springer-Verlag,
  Berlin-New York, 1972. \MR{0365573}

\bibitem[Bou79]{bousfield-localization-spectra}
A.~K. Bousfield, \emph{The localization of spectra with respect to homology},
  Topology \textbf{18} (1979), no.~4, 257--281. \MR{551009}

\bibitem[Bou87]{bousfield}
\bysame, \emph{On the homology spectral sequence of a cosimplicial space},
  Amer. J. Math. \textbf{109} (1987), no.~2, 361--394. \MR{882428}

\bibitem[Dug14]{dugger}
Daniel Dugger, \emph{A primer on homotopy colimits}, preprint, November, 2014.

\bibitem[Dwy98]{Dwyer-Sharp}
W.~G. Dwyer, \emph{Sharp homology decompositions for classifying spaces of
  finite groups}, Group representations: cohomology, group actions and topology
  ({S}eattle, {WA}, 1996), Proc. Sympos. Pure Math., vol.~63, Amer. Math. Soc.,
  Providence, RI, 1998, pp.~197--220. \MR{1603159 (99b:55033)}

\bibitem[Elm83]{Elmendorf}
A.~D. Elmendorf, \emph{Systems of fixed point sets}, Trans. Amer. Math. Soc.
  \textbf{277} (1983), no.~1, 275--284. \MR{690052}

\bibitem[LMM81]{Lewis-May-McClure}
G.~Lewis, J.~P. May, and J.~McClure, \emph{Ordinary {$RO(G)$}-graded
  cohomology}, Bull. Amer. Math. Soc. (N.S.) \textbf{4} (1981), no.~2,
  208--212. \MR{598689}

\bibitem[May96]{May-Alaska}
J.~P. May, \emph{Equivariant homotopy and cohomology theory}, CBMS Regional
  Conference Series in Mathematics, vol.~91, Published for the Conference Board
  of the Mathematical Sciences, Washington, DC; by the American Mathematical
  Society, Providence, RI, 1996, With contributions by M. Cole, G.
  Comeza{\~n}a, S. Costenoble, A. D. Elmendorf, J. P. C. Greenlees, L. G.
  Lewis, Jr., R. J. Piacenza, G. Triantafillou, and S. Waner. \MR{1413302}

\bibitem[Oli94]{Oliver-p-stubborn}
Bob Oliver, \emph{{$p$}-stubborn subgroups of classical compact {L}ie groups},
  J. Pure Appl. Algebra \textbf{92} (1994), no.~1, 55--78. \MR{1259669
  (94k:57055)}

\bibitem[Pia91]{Piacenza}
Robert~J. Piacenza, \emph{Homotopy theory of diagrams and {CW}-complexes over a
  category}, Canad. J. Math. \textbf{43} (1991), no.~4, 814--824. \MR{1127031}

\bibitem[Rie14]{Riehl-Categorical}
Emily Riehl, \emph{Categorical homotopy theory}, New Mathematical Monographs,
  vol.~24, Cambridge University Press, Cambridge, 2014. \MR{3221774}

\bibitem[Sch18]{Schwede-Global}
Stefan Schwede, \emph{Global homotopy theory}, New Mathematical Monographs,
  vol.~34, Cambridge University Press, Cambridge, 2018. \MR{3838307}

\bibitem[Shu09]{Shulman}
Michael Shulman, \emph{Homotopy limits and colimits and enriched homotopy
  theory}, preprint, Oct. 2006, last version Jul. 2009.

\bibitem[SS03]{Schwede-Shipley}
Stefan Schwede and Brooke Shipley, \emph{Equivalences of monoidal model
  categories}, Algebr. Geom. Topol. \textbf{3} (2003), 287--334 (electronic).
  \MR{1997322}

\bibitem[Ste16]{Stephan}
Marc Stephan, \emph{On equivariant homotopy theory for model categories},
  Homology Homotopy Appl. \textbf{18} (2016), no.~2, 183--208. \MR{3551501}

\bibitem[tD87]{MR889050}
Tammo tom Dieck, \emph{Transformation groups}, De Gruyter Studies in
  Mathematics, vol.~8, Walter de Gruyter \& Co., Berlin, 1987. \MR{889050}

\bibitem[Wan80]{waner}
Stefan Waner, \emph{Equivariant fibrations and transfer}, Trans. Amer. Math.
  Soc. \textbf{258} (1980), no.~2, 369--384. \MR{558179}

\bibitem[Wil75]{Willson}
Stephen~J. Willson, \emph{Equivariant homology theories on {$G$}-complexes},
  Trans. Amer. Math. Soc. \textbf{212} (1975), 155--171. \MR{0377859}

\end{thebibliography}

\newcommand{\etalchar}[1]{$^{#1}$}
\providecommand{\bysame}{\leavevmode\hbox to3em{\hrulefill}\thinspace}
\providecommand{\MR}{\relax\ifhmode\unskip\space\fi MR }
\providecommand{\MRhref}[2]{%
  \href{http://www.ams.org/mathscinet-getitem?mr=#1}{#2}
}
\providecommand{\href}[2]{#2}

\end{document}